\documentclass[11pt]{amsart}
\usepackage{amssymb,latexsym} \textwidth 15.00cm \textheight 22cm
\newtheorem{theorem}{Theorem}

\renewcommand{\theequation}{\arabic{section}.\arabic{equation}}
\renewcommand{\thetheorem}{\arabic{section}.\arabic{theorem}}

\newtheorem{remark}[theorem]{Remark}
\newtheorem{proposition}[theorem]{Proposition}
\newtheorem{lemma}[theorem]{Lemma}

\newtheorem{corollary}[theorem]{Corollary}

\newtheorem{conjecture}[theorem]{Conjecture}

\newcommand{\kh}{K\"{a}hler}

\newcommand{\p}{\partial}

\newcommand{\kl}{k\bar{l}}
\newcommand{\be}{\begin{equation}}
\newcommand{\ee}{\end{equation}}
\newcommand{\bp}{\begin{proposition}}
\newcommand{\ep}{\end{proposition}}
\newcommand{\bt}{\begin{theorem}}
\newcommand{\et}{\end{theorem}}

\begin{document}
\setlength{\baselineskip}{1.2\baselineskip}

\title[On a class of fully nonlinear flows in K\"{a}hler geometry]{On a class of fully nonlinear flows in K\"{a}hler geometry}
\author{Hao Fang}
 \address{Department of Mathematics\\
         University of Iowa, Iowa City, IA 52245}
\email{haofang@math.uiowa.edu}

\author{ Mijia Lai}
\address{Department of Mathematics\\
         University of Iowa, Iowa City, IA 52245}
         \email{mijlai@math.uiowa.edu}

\author{ Xinan Ma}
\address{Department of Mathematics\\
         University of Science and Technology of China\\
         Hefei, 230026, Anhui Province, CHINA}
         \email{xinan@ustc.edu.cn}

\thanks{The work of the first-named author is partially supported by NSF grant number DMS 060672. The work of the third-named author is partially supported by NSFC grant number 10671186, NSF grant number DMS 0635607 and Z\"urich Financial Services.}

\begin{abstract}
In this paper, we study a class of fully nonlinear metric flow on
K\"{a}hler manifolds, which includes the J-flow as a special case.
We provide a sufficient and necessary condition  for the long time
convergence of the flow, generalizing the result of Song-Weinkove.
As a consequence, under the given condition, we solved the
corresponding Euler equation, which is fully nonlinear of
Monge-Amp\`{e}re type. As an application, we also discuss a
complex Monge-Amp\`{e}re type equation including terms of mixed
degrees, which was first posed by Chen.
\end{abstract} \maketitle

\section{Introduction}
\setcounter{equation}{0}

In the study of K\"ahler geometry, the geometric flow method has
been applied extensively to obtain ''optimal'' metrics. One
classical example is the K\"ahler-Ricci flow. If the manifold has
negative or vanishing first Chern class, the K\"ahler-Ricci flow
converges to the Einstein metric, see Cao~\cite{C}. Another
example is the so-called J-flow. It was introduced by
Donaldson~\cite{D} in the setting of moment maps and by Chen
in~\cite{C1,C2}, as the gradient flow of the
\textsl{J}-functional, which appears as a term of the Mabuchi
energy. In~\cite{W1},  Weinkove settled the question of Donaldson for surfaces.   A sufficient class condition for the convergence of the J-flow is derived in~\cite{W2}. In~\cite{S-W}, Song and Weinkove proved a positivity condition to be equivalent to the convergence of the
J-flow to a critical metric; The precise statement
of this condition can be found in the discussion after (\ref{C_k}). In general, the solution of
these geometric flows usually depends on establishing a priori
estimates of parabolic PDEs.

In this paper, we will study a class of fully non-linear geometric
flows, which was motivated by the construction of J-flow.

Let $(M,\omega)$ be a closed K\"ahler manifold of dimension $n$.
Define \be \mathcal{H}^{+}= \{[\chi]\in H^{1,1}(M),\ \ \ \exists
\chi\in[\chi],\ \chi>0\ \}.\ee Let $[\chi]\in\mathcal{H}^{+}$ and
$\chi_{_0}\in [\chi]$ is another K\"ahler form on $M$.  We define
the corresponding K\"ahler cone and K\"ahler potential space with
respect to $[\chi]$ as
\be
\mathcal{K}_{[\chi]}=\{\chi_{\varphi}=\chi_{_0}+\frac{\sqrt{-1}}{2}\partial\overline{\partial}\varphi
>0,\ \ \ \varphi\in C^{\infty}(M) \},
\ee \be \mathcal{P}_{\chi_{_0}}=\{\varphi\in C^{\infty}(M)\mid
\chi_{\varphi}=\chi_{_0}+\frac{\sqrt{-1}}{2}\partial\overline{\partial}\varphi
>0\}.
\ee

For a fixed integer $k\in[1,n]$,  and
$\lambda=(\lambda_{1},\cdots, \lambda_{n})\in \mathbb{R}^n$, the
$k$-th elementary symmetric polynomial of $\lambda$ is defined as
\begin{eqnarray*}
\sigma_{0}(\lambda)&=&\ \ \ \ \ \ \ \ \ 1;\\
\sigma_k(\lambda)&=&\sum_{1\leq i_1<i_2<\cdots<i_k\leq n}\lambda_{i_1}\lambda_{i_2}\cdots\lambda_{i_k},\ \ \ k\geq 1.
\end{eqnarray*}
When no confusion arises, we also use $\sigma_k(A)$ to denote the
$k$-th elementary symmetric function of eigenvalues of a Hermitian
matrix $A$.

In a local normal coordinate system of $M$ with respect to
$\omega$, we have
$$\chi_{_0}=\frac{\sqrt{-1}}{2}\chi_{_{0i\bar{j}}}dz^{i}\wedge dz^{\bar{j}},\
\ \ \ \
\chi_{\varphi}=\frac{\sqrt{-1}}{2}(\chi_{_0i\bar{j}}+\varphi_{i\bar{j}})dz^{i}\wedge
dz^{\bar{j}}.$$ Following the notation above, we denote
$$\sigma_k(\chi_\varphi)={{n}\choose{k}}\frac{\chi_\varphi^k\wedge\omega^{n-k}}{\omega^n},$$
which is just the k-th elementary symmetric polynomial of the
eigenvalues of the matrix $(\chi_{_0i\bar{j}}+\varphi_{i\bar{j}})$
with respect to the background metric $\omega$.

We set the volume form on $M$ as $dv=\omega^{n}/n!$. It is clear that
\begin{eqnarray*}
c_{k}&=&c_{k,[\omega],[\chi]}=\frac{\int_{M}
\chi_{_0}^{n-k}\wedge\omega^k}{\int_{M} \chi_{_0}^n},\\
c'_{k}&=&c'_{k,[\omega],[\chi]}=\frac{{n\choose k}\int_{M}
\chi_{_0}^{n-k}\wedge\omega^k}{\int_{M} \chi_{_0}^n}=\frac{\int_{M}
\sigma_{n-k}(\chi_{_0}) \ dv}{\int_{M} \sigma_n(\chi_{_0})\  dv}
\end{eqnarray*} are topological constants.
Now we consider following flow in $\mathcal{P}_{\chi_{_0}}$:
\begin{eqnarray} \chi_{t}&=&\chi_{_0}+\frac{\sqrt{-1}}{2}\partial\overline{\partial}\varphi_{t},\notag\\\label{flow}
\frac{\partial\varphi_t}{\partial t}&=&{c'_{k}}^{\frac{1}{k}}
-(\frac{\sigma_{n-k}(\chi_{\varphi_t})}{\sigma_n(\chi_{\varphi_t})})^{\frac{1}{k}},\\
\varphi_{_0}&=&0. \notag
\end{eqnarray}  Clearly, the stationary
metric of this flow is a K\"{a}hler metric $\chi\in\mathcal{K}_{[\chi]}$ satisfying:
\begin{equation}
\chi^{n-k}\wedge\omega^k=c_{k}\chi^{n}=(\frac{\int
\chi_{_0}^{n-k}\wedge\omega^k}{\int \chi_{_0}^n}
)\chi^n.\label{euler}
\end{equation}

In the case of $k=1$, our flow is same as the $\textsl{J}$-flow. Song-Weinkove~\cite{S-W} gave a sufficient and necessary condition for the J-flow to exist and converge to a solution of
(\ref{euler}).

One of the purposes of this paper is to give a necessary and
sufficient condition for the flow (\ref{flow}) to converge to the
stationary metric, which we now describe as a cone condition. For
$M$ and $\omega$ given as above, we define $ {\mathcal
C}_{k}={\mathcal C}_{k}(\omega)$ as

\be {\mathcal C}_{k}(\omega)=\{[\chi]\in\mathcal{H}^{+}, \ \exists
\chi'\in[\chi], \ s.t.\
c_{k}n\chi'^{n-1}>(n-k)\chi'^{n-k-1}\wedge\omega^k\}. \label{C_k}
\ee

${\mathcal C}_{k}$ is an affine cone in $\mathcal{H}^{+}$.  For $k=1$, $C_{1}$ is first defined in~\cite{S-W}. It is
easy to check that $[\chi]\in {\mathcal C}_{k}$ is a necessary
condition for  the equation (\ref{euler}) to be solvable (see
Section 2 for more details).  The main theorem of this paper is
the following

\begin{theorem}
Suppose $M$, $\omega$ and $\chi_{_0}\in[\chi]$ are defined as
above. Let $1\leq k \leq n$. If $[\chi]\in
\mathcal{C}_{k}(\omega)$, then flow (\ref{flow}) has a long time
solution, which converges to a smooth metric satisfying
(\ref{euler}).\label{theorem1}
\end{theorem}

It is worthwhile to point out the case of  $k=n$.  Notice that the
corresponding equation is equivalent to
\be
\chi_{\varphi}^n=\frac{\int_{M} \chi_{_0}^n}{\int_{M}
\Omega}\Omega,\label{1.14}
\ee where $\Omega$ is any given volume
form. This was solved by Yau in his celebrated paper~\cite{Y}.
Also notice that the condition  (\ref{C_k}) becomes trivial in
this case; in other words, ${\mathcal C}_{n}={\mathcal H}^{+}$.
Cao~\cite{C} provides a parabolic approach to this equation, using
Ricci flow.

Notice that for the $k=1$ case, our condition and conclusion are
exactly same as the ones in~\cite{S-W}.

Theorem~\ref{theorem1} can be viewed as a finite interpolation
between results of Yau~\cite{Y}, Cao~\cite{C},
Song-Weinkove~\cite{S-W}. In fact, our basic approach to prove
Threorem~\ref{theorem1} closely follows these earlier works. In particular, the idea of establishing partial $C_{0}$ estimate before $C_{2}$ and $C_{0}$ estimates first appears in~\cite{W1}. However, new convexity phoneomena shows up for $k\neq 1,n$ cases.

Theorem~\ref{theorem1} can be understood from several aspects.

First, Theorem~\ref{theorem1} can be understood geometrically. One
motivation for the construction of this flow (\ref{flow}), as well
as an important ingredient of the proof of Theorem~\ref{theorem1},
is the following functional defined for $\chi_{\phi}$ with
${\phi}\in\mathcal{P}_{\chi_{_0}}$ and $j\geq 0$,
 \be {\mathcal
F}_{j}(\chi_{\phi})=\int_0^{1}\int_{M}\frac{\p\phi_t}{\p
t}\chi_{\phi_t}^j\wedge\omega^{n-j} dt, \ee where $\phi_{t}\in
\mathcal{P}_{\chi_{_0}}$, $t\in[0,1]$ is a path in connecting
$\chi_{_0}$ and $\chi_{\phi}$. ${\mathcal F}_{j}$ is shown to be
independent of the choice of path \cite{ChT}. Furthermore, a
functional defined as \be\mathcal{\tilde
F}_{j,n}(\chi_{_0},\chi_{\phi})= {\mathcal F}_{j}(\chi_{\phi})-
c_{n-j} {\mathcal F}_{n}(\chi_{\phi})\ee can be viewed as a
functional depending only on $\chi_{_0},\chi_{\phi}\in
\mathcal{K}_{\chi}$.

Notice that for $\chi_{_i}\in[\chi]$, $i=0,1,2$, we have
$$\mathcal{\tilde F}_{j,n}(\chi_{_0},\chi_{_1})+\mathcal{\tilde F}_{j,n}(\chi_{_1},\chi_{_2})=\mathcal{\tilde F}_{j,n}(\chi_{_0},\chi_{_2}).$$
Thus, the minimizer of functional $\mathcal{\tilde
F}_{j,n}(\chi_{0},\cdot)$ is independent of the choice
$\chi_{_0}$. In fact, this functional can be realized as quotients
of Quillen metrics on the determinant bundles with certain virtual
bundle coefficients, see Tian~\cite{T1}.

Our flow (\ref{flow}) is constructed in such a way that the
functional $\mathcal{\tilde F}_{n-k,n}
(\chi_{_0},\chi_{\varphi_{t}})$ is decreasing along the flow. It
is then easy to check that  the corresponding minimum metric
satisfies (\ref{euler}).  Theorem~\ref{theorem1} gives an explicit
path for the functional  $\mathcal{\tilde F}_{n-k,n}
(\chi_{_0},\chi)$ to obtain its unique minimal, when the cone condition
$[\chi]\in \mathcal{C}_{k}$ is satisfied. Notice that our flow is
not the gradient flow of the corresponding functionals except the
case $k=1$. In fact, we modified the functional's gradient flow to
ensure certain PDE estimates hold.

Second, Theorem~\ref{theorem1} provides a necessary and sufficient
condition for (\ref{euler}), an elliptic equation of
Monge-Amp\`{e}re type to be solvable. Notice that  (\ref{euler})
can be written, locally, for $k<n$ as \be
c'_{k}\sigma_{n}(\chi_{\varphi})=\sigma_{n-k}(\chi_{\varphi}), \ee
or, equivalently,
$$\sigma_{k}(\chi_{\varphi}^{-1})=c'_{k}.$$ The corresponding $[\chi]\in\mathcal{C}_{k}$ condition states that there exists a $\chi'\in[\chi]$ such that
\be\sigma_{k}(\chi'^{-1}| i)<c'_{k},\ee  for $1\leq i \leq n$. Refer to Section 2 for
more details.

Equation (\ref{euler}) is also a special case of a question posed
by Chen. In~\cite{C1}, Chen raised the question of solving a very
general fully non-linear equation of Monge-Amp\`{e}re type:
\be\label{cheneqn}
\chi_{\varphi}^n=\sum_{i=0}^{n-1}\alpha_i\chi_{\varphi}^i\wedge
\omega^{n-i},\ee where $\alpha_{i}$'s are real.
Theorem~\ref{theorem1} gives a complete answer for Chen's question
when the right hand side has only one term.

Using similar method, we can also extend our result.

Define, for any fixed  $\alpha\in(0,\infty)$ and  integer $k\in [1,n]$,
\begin{eqnarray*} c_{k,\alpha}&=&c_{k,\alpha,[\omega],[\chi]}= c_{k}+\alpha c_{k-1} ,\\
\mathcal {\tilde F}_{\alpha,k,n}(\chi_{_0},\chi)&=&  \mathcal {\tilde F}_{n-k,n}(\chi_{_0},\chi)+\alpha \mathcal {\tilde F}_{n-k+1,n}(\chi_{_0},\chi),\\
 {\mathcal C}_{k,\alpha}(\omega)&=&\{[\chi]\in\mathcal{H}^{+}, \ \exists \chi'\in[\chi], \ {\rm such \ that}\\&&
 c_{k,\alpha} n\chi'^{n-1}> (n-k)\chi'^{n-k-1}\wedge\omega^k+\alpha
 (n-k+1)\chi'^{n-k}\wedge\omega^{k-1}\}.\label{ckalpha}
\end{eqnarray*}
It is clear to see that when the parameter $\alpha$ runs from $0$
to $\infty$, ${\mathcal C}_{k,\alpha}={\mathcal
C}_{k,\alpha}(\omega)$ gives a continuous deformation from the
cone $\mathcal{C}_{k}\subset \mathcal{H}^{+}$ to
$\mathcal{C}_{k-1}\subset \mathcal{H}^{+}$. We have the following

\bt \label{theorem2} Suppose $M$, $\omega$ and
$\chi_{_0}\in[\chi]$ are defined as above. Assume $1\leq k\leq n$
and $\alpha>0$, then the equation \be c_{k,\alpha}
\chi^{n}=\chi^{n-k}\wedge\omega^k+\alpha
\chi^{n-k+1}\wedge\omega^{k-1} \ee has a unique smooth solution if
and only if $[\chi]\in \mathcal{C}_{k,\alpha}(\omega)$; in this
case, the solution minimizes $\mathcal {\tilde
F}_{\alpha,k,n}(\chi_{_0},\chi)$. \et

Theorem~\ref{theorem2} is proved by  improving the estimates
needed in proving Theorem~\ref{theorem1} to the product manifold
$M\times C$, where $C$ is a smooth algebraic curve.

Based on these known results, we would like to verify that the
similar cone condition would be the necessary and sufficient
condition for the problem of Chen. Using a similar geometric
construction as in the proof of Theorem~\ref{theorem2}, we can
settle many special cases for Chen's problem.  See Section 5 for
more details. We believe this is one of the few examples of the
Monge-Amp\`{e}re type  equations including terms of mixed degrees.
The  geometric structure plays an important role in the solution
of these equations.

Finally, we make some remarks.
\begin{remark}
It is  interesting to point out that the elliptic PDEs studied in
this paper are all solved by geometric flow method. With the
exception of Yau's original equation, continuity method does not
seem to work for the other cases.  \end{remark}

\begin{remark}
It is interesting to study the various cones we defined in
$\mathcal{H}^{+}$. Except the obvious fact that ${\mathcal
C}_{n}(\omega)={\mathcal H}^{+}$ includes all the other cones, the
relative position of ${\mathcal C}_{j}(\omega)$ and ${\mathcal
C}_{k}(\omega)$ for $j\neq k$, $j,k\neq n$ is unknown.
\end{remark}

\begin{remark}
The strong concavity property of the symmetric polynomials is very
important for our estimates. We point out that we do not use the
optimal concavity property available. This leaves room of future
construction of other geometric flows in K\"ahler geometry.
\end{remark}

The rest of this paper is organized as follows. In Section 2 we
introduce further notation and some preliminary facts about the
elementary symmetric polynomials. In Section 3, we derive the
partial $C^2$ estimate by maximum principle, following
Yau~\cite{Y} and Weinkove~\cite{W1}. In Section 4, we derive the
$C^{0}$ estimate and $C^{\infty}$ estimate and the convergence
result. In section 5, we discuss various generalization of
Theorem~\ref{theorem1} and some application to complex geometry.
In the Appendix, we give an alternative proof of our strong
concavity property.

{Acknowledgments:  The first-named author would like to thank Jian
Song for useful discussion. All authors would like to thank
Pengfei Guan and Lihe Wang for discussion. They would like to thank Institute
for Advanced Study for support and hospitality. Most of this work
is done when they  attended special year of Geometric non-linear
PDE at IAS. Thanks also go to referee for his or her careful
proof-reading and useful suggestion.}
\section{Preliminary}
\setcounter{equation}{0} \setcounter{theorem}{0} In this section,
we set up the notation and prove some preliminary results
regarding elementary symmetric functions.

For simplicity, after proper scaling, we may assume
$c_k=\frac{\int \chi_{_0}^{n-k}\wedge\omega^k}{\int
\chi_{_0}^n}=1$ without loss of generality. We also denote
$c=c'_{k}={n \choose k}$ when no confusion occurs.

Fix a local coordinate chart $U\subset M$. For $z=(z_1,z_2, \cdots
, z_n)\in U$, we write
\begin{eqnarray*}
\omega&=&\frac{\sqrt{-1}}{2}g_{i\bar{j}}dz^i\wedge dz^{\bar{j}},\\
\chi_{_0}&=&\frac{\sqrt{-1}}{2}\chi_{_{0i\bar{j}}}dz^i\wedge dz^{\bar{j}},\\
\chi'&=&\frac{\sqrt{-1}}{2}\chi'_{i\bar{j}}dz^i\wedge dz^{\bar{j}},\\
\chi_{\varphi}&=&\frac{\sqrt{-1}}{2} (\chi_{_{0i\bar{j}}}+\varphi_{i\bar{j}})dz^i\wedge dz^{\bar{j}},\\
\chi_{i\bar{j}}&=&\chi_{_{0i\bar{j}}}+\varphi_{i\bar{j}}.
\end{eqnarray*}

When no confusion occurs, we also use $\chi_{_0}$, $\chi'$,
$\chi_{\varphi}$ to denote the corresponding Hermitian matrices at
the given $z$. We always choose the normal coordinate of $\omega$
such that $g_{i\bar{j}}=\delta_{i\bar{j}}$ and $\chi_{\varphi}$ is
diagonal. In other words, we have
$\chi_{\varphi}(z)=\chi=(\chi_{_1},\cdots,\chi_{_n})$.
Furthermore, we may assume $\chi_{_i}\geq\chi_{_j}$ for $i>j$.
That means $\chi_{_1}$ and $\chi_{_n}$ are the maximal and the
minimal eigenvalues of $\chi_{\varphi}$, respectively.

For a Hermitian matrix $A=(a_{i \bar j})_{n\times n}$, define
$$F(A):=-[\frac{\sigma_{n-k}(A)}{\sigma_n(A)}]^{^{\frac{1}{k}}}=-\sigma_k^{\frac{1}{k}}(A^{-1}).$$
It is a well known fact that F is a concave function of $A$ and
${F^{i\bar{j}}}$ is positive definite when restricted to the space
of positive definite hermitian matrixes (see e.g., \cite{S}).
Without further specification, we assume that $A$ is positive in
the rest of this section.

We compute the derivatives of $F$ with respect to entries of $A$
for the future use. \bp For $F$  given as above, we have
$$ F^{i\bar{j}}(A):=\frac{\p
F}{\p
a_{i\bar{j}}}=-\frac{1}{k}(\frac{\sigma_{n-k}}{\sigma_n})^{1/k-1}(\frac{\p
\sigma_{n-k}/\p
a_{i\bar{j}}}{\sigma_{n}}-\frac{\sigma_{n-k}\p\sigma_{n}/\p
a_{i\bar{j}}}{\sigma_{n}^2}).$$
$$F^{i\bar{j},\kl}(A):=\frac{\p^2 F(A)}{\p a_{i\bar{j}}\p a_{\kl}}.$$

If $A=\chi={\rm diag} (\chi_1,\chi_2, \cdots, \chi_n)$ is
diagonal, then $F^{i\bar{j}}$ can be non trivial $\ \text{iff} \ \
i=j$. We have
$$F^{i\bar{i}}=-\frac{1}{k}(\frac{\sigma_{n-k}(\chi)}{\sigma_n(\chi)})^{1/k-1}(\frac{\sigma_{n-k-1}(\chi|i)}{\sigma_{n}(\chi)}-\frac{\sigma_{n-k}(\chi)\sigma_{n-1}(\chi|i)}{\sigma_{n}^{2}(\chi)}),$$
or
$$F^{i\bar{i}}=\frac{1}{k}\sigma_k^{\frac{1}{k}-1}({\chi}^{-1})\sigma_{k-1}({\chi}^{-1}|i)\frac{1}{\chi_i^2}.$$
Furthermore,  $F^{i\bar{j},\kl}$ can be nontrivial  $\ \text{iff}\
i=j, k=l \ \text{or}\  i=l,j=k$. In this case, we have
$$F^{i\bar{j},j\bar{i}}(\chi)=\frac{1}{k}(\frac{\sigma_{n-k}(\chi)}{\sigma_n(\chi)})^{\frac{1}{k}-1}(\frac{\sigma_n(\chi)\sigma_{n-k-2}(\chi|i,j)-\sigma_{n-k}(\chi)\sigma_{n-2}(\chi|i,j)}{\sigma_{n}^2(\chi)}), \text{for}\  i\neq j,$$
where $\chi^{-1}$ denotes the inverse matrix of $\chi$,
$\sigma_k(\chi|i)=\sigma_k(\chi)|_{\chi_i=0}$,
$\sigma_k(\chi|i,j)=\sigma_k(\chi)|_{\chi_i=0,\chi_j=0}$.
\label{p21}\label{derivative}\ep

Also notice that $F$ is homogenous of degree $-1$, so $-F(A)=\sum
_{i,j}F^{i\bar{j}}(A)a_{i\bar{j}}$.\\

We proceed to discuss some technical results. First of all, we
have the following concavity result. Define
$$\Gamma_n=\{(x_1,\cdots, x_n)\in \mathbb{R}^n|x_1>0,x_2>0,\cdots
x_n>0\}.$$
\begin{proposition}\label{Prop2.a}[GLZ]
Let $g(\lambda)= \log\sigma_k(\lambda)$. For $\lambda \in
\Gamma_n$,  $\xi=(\xi_{1},\cdots,\xi_{n}) \in \mathbb{C}^{n}$, let $g_i:=\frac{\p g}{\p
\lambda_i}$, $g_{ij}:=\frac{\p^{2} g}{\p \lambda_i \p \lambda_j}$, we
have
\begin{eqnarray}
\sum_{i=1}^n (g_{ii}+ \frac{g_i}{\lambda_{i}})\xi_i\bar\xi_i +
\sum_{i\neq j}g_{ij}\xi_i\bar\xi_j
 \ge 0.
\label{2.a}
\end{eqnarray}
\end{proposition}
\begin{proof}
We have
$$g_i = \frac{\sigma_{k-1}(\lambda|i)}{\sigma_k(\lambda)}, \quad g_{ij} =
 \frac{\sigma_{k-2}(\lambda|i,j)}{\sigma_k(\lambda)}- \frac{\sigma_{k-1}(\lambda|i)\sigma_{k-1}(\lambda|j)}{\sigma_k^2(\lambda)}.$$
Using the same reduction in Lemma 2.3 of ~\cite{GM}, (\ref{2.a})
can be reduced to the following inequality
\begin{eqnarray}
\sum_{i=1}^n\sigma_k(\lambda|i)\sigma_{n-1}(\lambda|i)\sigma_{k-1}(\lambda|i)|\xi_i|^2 \nonumber\\
\ge \sigma_{n}(\lambda)\sum_{i \neq
j}\{\sigma_{k-1}^2(\lambda|ij)-\sigma _{k}(\lambda|i
j)\sigma_{k-2}(\lambda|ij)\}\xi_i\bar\xi_j , \label{2.c}
\end{eqnarray}
which is just Lemma 2.4 in \cite{GM}.
\end{proof}
\begin{remark}\label{remk2.0}
By the above proposition, if we let $g(\lambda)=
\sigma_k^{\frac{1}{k}}(\lambda)$, for $\lambda \in \Gamma_n$, then
a simple calculation shows, for $\xi=(\xi_{1},\cdots,\xi_{n}) \in \mathbb{C}^{n}$,
\begin{eqnarray}
(g_{ij}+ \frac{g_{i}}{\lambda_j}{\delta}_{ij})\xi_{i}\bar{\xi_{j}} \ge 0.\nonumber
\end{eqnarray}
Another proof will be given in the appendix.
\end{remark}

Second, we have the following local version of the cone condition
(\ref{C_k}).

\bp For $k<n$,$\chi'\in\mathcal{C}_{k}$ is equivalent to
$$\frac{\sigma_{n-k-1}(\chi'|j)}{\sigma_{n-1}(\chi'|j)}=\sigma_k(\chi'^{-1}|j)<{n\choose
k}, $$  for any $j\in\{1,\cdots,n\}$, where $(\chi'|j)$ denotes the matrix obtained
by deleting the j-th column and j-th row of
$\chi'$.\label{new2.2}\ep

\begin{proof}
Assume $\chi'\in\mathcal{C}_{k}$. By (\ref{C_k}), for any given
integer $j\in[1,n]$, the coefficient of the $(n-1,n-1)$ form
${\Pi_{i=1}^{n}}_{i\neq j} dz^{i}d\bar{z}^{i}$ in
$\chi'^{n-1}-\frac{n-k}{n}\omega^k\wedge\chi'^{n-k-1}$ should be
positive; that is,
$$(n-1)!\sigma_{n-1}(\chi'|j)-\frac{n-k}{n}k!(n-k-1)!\sigma_{n-k-1}(\chi'|j)>0.$$ Dividing both sides by
$\frac{n-k}{n}k!(n-k-1)!\sigma_{n-1}(\chi'|j)$, one obtains
$$\frac{\sigma_{n-k-1}(\chi'|j)}{\sigma_{n-1}(\chi'|j)}<{n\choose
k}.$$
\end{proof}

Next, we introduce some simple algebraic facts. Let
$A=(a_{i\bar{j}})$ be a positive Hermitian matrix.

\begin{lemma} Let $I=(i_1,i_2,\cdots, i_k)\subset(1,\cdots, n)$ be an index set, denote its
complement in $(1,2,\cdots, n)$ by $\bar{I}$. We always order
$\bar{I}$ so that $(I,\bar{I})$ is an even permutation of
$(1,2,\cdots,n)$. For $A$, a positive hermitian $n\times n$
matrix, let $A_I$ be the principal minor
 $(a_{i\bar{j}})_{i,j\in I}$. Then $$ \det(A)\leq
\det(A_I)\det(A_{\bar{I}}).$$ \label{lemma2.6}\end{lemma}
\begin{proof} Rearrange $A$ if necessary we may write $A$ as
\begin{equation}
A=
\begin{bmatrix}
  A_I & M \\
  M' & A_{\bar{I}}
\end{bmatrix}.
\end{equation}
By
\begin{equation}
\begin{bmatrix}
  Id & 0 \\
  -M'A_I^{-1} & Id
\end{bmatrix}
\begin{bmatrix}
  A_I & M \\
  M' & A_{\bar{I}}
\end{bmatrix}=
\begin{bmatrix}
 A_I & M \\
  0 & A_{\bar{I}}-M'A_I^{-1}M \\
\end{bmatrix},
\end{equation}
one obtains
$$ \det(A)=\det(A_I)\det(A_{\bar{I}}-M'A_I^{-1}M)\leq
\det(A_I)\det(A_{\bar{I}}),$$ where $M'$ means the conjugate
transpose matrix of $M$. The last inequality follows from the fact
that $M'A_I^{-1}M$ is positive definite.
\end{proof}
The following corollary is a direct consequence of
Lemma~\ref{lemma2.6}.
\begin{corollary}
Let $A$ be as above. Then $\det(A)\leq
\prod_{i=1}^{n}a_{i\bar{i}}$.\label{cor2.6}
\end{corollary}\qed

We are then ready to prove

\begin{lemma} Let
$A=(a_{i\bar{j}})$ be a positive Hermitian matrix. Denote ${\tilde
A}=(a_{i\bar{j}}\delta_{i\bar{j}})$ to be the matrix containing
only the diagonal terms of $A$.
 We have,
\begin{equation}
\sigma_k(\tilde A^{-1})\leq \sigma_k(A^{-1}).
\end{equation}\label{lemma2.5}
\end{lemma}
  \proof By
Corollary~\ref{cor2.6}, we have
$${\frac{1}{\det ({\tilde A})}}\leq \frac{1}{\det (A)}.$$
This means that Lemma~\ref{lemma2.5} holds for $k=n$. For general
$k$, we have
\begin{eqnarray*}
\sigma_k({\tilde A}^{-1})&=& \sum_{|I|=k,(i_1,i_2,\cdots,
i_k)\in I}
\frac{1}{a_{i_1\bar{i_1}}}\frac{1}{a_{i_2\bar{i_2}}}\cdots
\frac{1}{a_{i_k\bar{i_k}}}\\
&\leq& \sum_{|I|=k} \frac{1}{\det(A_I)}\leq \sum_{|I|=k}
\frac{\det (A_{\bar{I}})}{\det
(A)}=\frac{\sigma_{n-k}(A)}{\sigma_n(A)}=\sigma_k(A^{-1}).
\end{eqnarray*}
\qed

Finally, we give the following technical statement, which will be
used in the next section.\bt \label{t2.8}Assume that
$M,\omega,\chi\in[\chi]$ given as before. Assume that $k<n$  and
$[\chi]\in\mathcal{C}_{k}$. Let  $F^{i\bar{i}}(\chi)$ be given as
in Proposition~\ref{derivative}. Let $\chi'\in[\chi]$ be the
K\"{a}hler form satisfying the condition of $\mathcal{C}_{k}$.
Assume $C_1\leq\frac{\sigma_{n-k}(\chi)}{\sigma_n(\chi)}\leq C_2,$
for some universal constants $C_1$ and $C_2$.  Then there exists a
universal constant $N$, depending only on the given geometric
data, such that, if $\frac{\chi_{_1}}{\chi_{_{n}}}\geq N$ then
there exists $\epsilon>0$ such that
\be(1-\epsilon)\sum_{i=1}^{n}F^{i\bar{i}}(\chi)\chi'_{i\bar{i}}\geq
c^{-\frac{1}{k}}\sigma_k^{\frac{2}{k}}(\chi^{-1}).\label{2.3}\ee\label{2.2}
\et

 \proof
Follow the convention, we will verify (\ref{2.3}) under normal
coordinates which diagonalizes $\chi$ at some point. So
$\chi=diag(\chi_1,\chi_2, \cdots \chi_n)$, and $\chi_1\geq \chi_2
\geq \cdots \geq \chi_n$. In local coordinates we will use
$\sigma_k(\chi^{-1})=\frac{\sigma_{n-k}(\chi)}{\sigma_n(\chi)}$
when no confusion arises.

We first notice for the case $\chi_n\ll 1$, (\ref{2.3}) follows
easily. Notice $\chi'$ is a fixed k\"{a}hler form, so there is a constant $\lambda>0$  such that
$$\chi'>\lambda \omega.$$ Therefore,
\begin{eqnarray}\label{2.5}
\sum_{i=1}^{n}F^{i\bar{i}}(\chi)\chi'_{i\bar{i}}&\geq &
\lambda\sum_{i=1}^{n}F^{i\bar{i}}(\chi)\\\notag &=&\lambda
\frac{1}{k}\sigma_k^{\frac{1}{k}-1}(\chi^{-1})\sum_{i=1}^{n}\sigma_{k-1}(\chi^{-1}|i)\frac{1}{\chi_{_i}^2}\\\notag
&\geq& \lambda
\frac{1}{k}\sigma_{k}^{\frac{1}{k}-1}(\chi^{-1})\sigma_{k-1}(\chi^{-1}|n)\frac{1}{\chi_{_n}^2}.\\\notag
\end{eqnarray}
We claim $\sigma_{k-1}(\chi^{-1}|n)\frac{1}{\chi_{_n}}$ is bounded
below. Indeed, $\sigma_{k-1}(\chi^{-1}|n)\frac{1}{\chi_{_n}}$ is
the largest term among
$\sigma_{k-1}(\chi^{-1}|i)\frac{1}{\chi_{_i}}$ by the fact that
$\chi_n$ is the smallest among $\chi_i,1\leq i \leq n$.
Thus,
 %\begin{equation}\sum_{i=1}^{n}\sigma_{k-1}(\chi^{-1}|i)\frac{1}{\chi_{_i}}=k
%\sigma_k(\chi^{-1}).\end{equation} So
\be\sigma_{k-1}(\chi^{-1}|n)\frac{1}{\chi_{_n}}\geq 1/n[\sum_{i=1}^{n}\sigma_{k-1}(\chi^{-1}|i)\frac{1}{\chi_{_i}}] =\frac{k}{n}\sigma_k(\chi^{-1})\label{new11}\ee

Now if $\chi_{n}<\delta=\lambda (C_{1}c)^{1\over k}$, (\ref{2.3}) follows easily from (\ref{2.5}) and (\ref{new11}).
%Now the RHS of (\ref{2.3}) is bounded from above by
%$c^{-\frac{1}{k}}C_2^{\frac{2}{k}}$, from (\ref{2.5})if we choose
%$\delta\leq \frac{\lambda}{n}(\frac{cC_1}{C_2})^{\frac{1}{k}}$
%then RHS of (\ref{2.5}) is greater than RHS of (\ref{2.3}). Hence
%(\ref{2.3}) follows easily.

So we just need to consider the case $\chi_n\geq \delta$.

Recall G\r{a}rding's inequality: For $\mu, \tau\in\Gamma_n$,
$$ \frac{1}{k}\sum_{j=1}^{n}\tau_j\frac{\p}{\p
\mu_j}\sigma_k(\mu)\geq
\sigma_k^{\frac{1}{k}}(\tau)\sigma_k^{1-\frac{1}{k}}(\mu).$$ Thus,
by Proposition~\ref{p21}, we have, for the matrix $B={\rm
diag}(\frac{\chi'_{1\bar{1}}}{\chi_1^2}, \cdots,
\frac{\chi'_{n\bar{n}}}{\chi_n^2})=\chi^{-1}\tilde{\chi'}\chi^{-1}$,
\begin{eqnarray}\label{2.4}
\sum_{i=1}^{n}F^{i\bar{i}}(\chi)\chi'_{i\bar{i}}&=&\sigma_k^{\frac{1}{k}-1}(\chi^{-1})\frac{1}{k}\sum_{i=1}^{n}\sigma_{k-1}({\chi^{-1}}|i)\frac{\chi'_{i\bar{i}}}{(\chi_i)^2}\\\notag
&\geq&\sigma_k^{\frac{1}{k}-1}({\chi^{-1}})\sigma_k^{1-\frac{1}{k}}({\chi^{-1}})\sigma_k^{\frac{1}{k}}(B)\\\notag
&=&\sigma_k^{\frac{1}{k}}(B).\\\notag
\end{eqnarray}
Comparing with (\ref{2.3}), it suffices to show
\begin{equation}\label{2.10}c^{\frac{1}{k}}\sigma_k^{\frac{1}{k}}(B)\geq
(1+\theta)\sigma_k^{\frac{2}{k}}(\chi^{-1}), \text{for}\ \ \
\theta>0.
\end{equation} By Proposition~\ref{new2.2}, we have
\begin{equation}\sigma_k((\chi'|1)^{-1})\leq{n \choose k}-\eta=c-\eta,
\label{2.7}
\end{equation}
for a universal positive constant $\eta<c$, depending only on
$(M,\omega)$ and $\chi'$,where $(\chi'|1)^{-1}$ is the inverse
matrix of $(\chi'|1)$. We have,
\begin{eqnarray}\label{2.8}
c^{\frac{1}{k}}\sigma_k^{\frac{1}{k}}(B)&\geq &
(\frac{c}{c-\eta})^{\frac{1}{k}}
\sigma_k^{\frac{1}{k}}((\chi'|1)^{-1})
\sigma_k^{\frac{1}{k}}(B)\\\notag &\geq&
(\frac{c}{c-\eta})^{\frac{1}{k}}
\sigma_k^{\frac{1}{k}}((\tilde{\chi'|1})^{-1})
\sigma_k^{\frac{1}{k}}(B) \\\notag &\geq &
(\frac{c}{c-\eta})^{\frac{1}{k}}
\sigma_k^{\frac{1}{k}}((\tilde{\chi'|1})^{-1})
\sigma_k^{\frac{1}{k}}(B|1) \\\notag &\geq &
(\frac{c}{c-\eta})^{\frac{1}{k}}\sigma_{k}^{\frac{2}{k}}(\chi^{-1}|1).
\end{eqnarray}
We explain the second and last inequality in (\ref{2.8}). Apply
Lemma~\ref{lemma2.5} to the matrix $(\chi'|1)$, we have
\begin{equation}\label{add1}
\sigma_k((\chi'|1)^{-1})\geq\sigma_k((\tilde{\chi'|1})^{-1}).
\end{equation}
Recall that $B=\chi^{-1}\tilde{\chi'}\chi^{-1}$, then
Cauchy-Schwarz inequality yields
$$\sigma_k( \chi^{-1}\tilde{\chi'}\chi^{-1}|1)\sigma_k((\tilde{\chi'|1})^{-1})\geq\sigma_k^{2}(\chi^{-1}|1).$$

Now suppose $\chi_{_1}\geq N\chi_{_n}$, and $\chi_{_n}\geq
\delta$. Then \begin{eqnarray}
\frac{\sigma_k(\chi^{-1}|1)}{\sigma_k(\chi^{-1})}&=&
1-\frac{\frac{1}{\chi_{_1}}\sigma_{k-1}(\chi^{-1}|1)}{\sigma_k(\chi^{-1})}\\\notag
&\geq& 1- \frac{\frac{1}{\chi_{_1}} \frac{{n-1\choose
k-1}}{\delta^{k-1}}}{\sigma_k(\chi^{-1})} \geq 1-
\frac{{n-1\choose k-1}}{C_1N \delta^k}.\end{eqnarray} Combine
(2.10),(2.12),(2.14), for $\theta$ sufficiently small, a  positive  number
$N=\frac{{n-1 \choose k-1}}{ C_1
\delta^{k}}\frac{1}{1-(1+\theta)^{\frac{1}{2}}(\frac{c-\eta}{c})}$
will satisfy the condition of this Proposition. \qed

\section{Partial Second order estimate}
\setcounter{equation}{0} \setcounter{theorem}{0} In this section,
we use the maximum principle to obtain an estimate on the second
order derivatives of $\varphi$ in terms of $\varphi$.

First we establish the ellipticity condition. Notice that by the
basic properties of symmetric polynomials, $(F^{i\bar{j}})>0$ if
$\chi>0$. Differentiating (\ref{flow}) with respect to $t$ gives
\begin{equation}
\frac{\p}{\p t}(\frac{\p \varphi}{\p
t})=F^{i\bar{j}}(\chi)\p_i\p_{\bar{j}}(\frac{\p \varphi}{\p
t}).\label{3.1}
\end{equation}
Standard theory for parabolic equation ensures short time
existence of the flow. By the maximum principle, $\frac{\p \varphi}{\p t}$
achieves extremal values at $t=0$, i.e.
 \be\min_{t=0} \frac{\p
\varphi}{\p t} \leq\frac{\p \varphi}{\p t}\leq \max_{t=0} \frac{\p
\varphi}{\p t},\label{new1}\ee
which in terms implies
\begin{equation}\label{3.2} \inf_M {\sigma_{n-k}\over\sigma_{n}}(\chi_{_0})\leq {\sigma_{n-k}\over\sigma_{n}}(\chi_{\varphi})\leq \sup_M
{\sigma_{n-k}\over\sigma_{n}}(\chi_{_0}).\end{equation} Hence,
$\chi_{\varphi}>0$, i.e., it remains K\"{a}hler when the flow exists.

Next we prove the partial $C^2$ estimate: \bt Let $M$, $\omega$,
and $\chi_{_0}\in[\chi]$ as above. $k$ is an integer in $[1,n]$.
Suppose $[\chi]\in \mathcal{C}_k$, i.e. there exists $\chi'\in
[\chi]$ such that:
$$\chi'^{n-1}-\frac{n-k}{n}\omega^k\wedge\chi'^{n-k-1}>0.$$
Let $\varphi$ be a solution of (\ref{flow}) on $[0,T)$. Then there
exist constants $A>0$,$C>0$, depending only on the initial data
and independent of $T$, such that for any time $t\geq0$,
$$\|\partial \bar{\partial} \varphi\|_{C^0}\leq Ce^{A(\varphi-\inf_{M\times[0,t]}\varphi)}.$$ \label{t3.1}\et

\begin{proof} By hypothesis, there exists $\phi\in \mathcal{P}_{\chi_{_0}}$, such that
$
\chi'=\chi_{_{0}}+\frac{\sqrt{-1}}{2}\partial\bar{\partial}\phi,$
then
$\chi_{\varphi}=\chi'+\frac{\sqrt{-1}}{2}\partial\bar{\partial}(\varphi-\phi)$.
Consider the function \begin{equation*}G(x,t,\xi):=\log
(\chi_{i\bar{j}}\xi^i\xi^{\bar{j}}) - A
(\varphi-\phi),\end{equation*} for $x\in M$, and $\xi\in
\mathbf{T}_{x}^{(1,0)}M$, $g_{i\bar{j}}\xi^{i}\xi^{\bar j}=1$. $A$
is a constant to be determined. Fix a time $t$, we can assume $G$
attains maximum at $(x_{_0},t_{_0})\in M\times [0,t]$, along the
direction $\xi_{_0}$. Choose normal coordinates of $\omega$ at
$x_{_0}$, so that $\xi_{_0}=\frac{\p}{\p z_1}$ and
$(\chi_{i\bar{j}})$ is diagonal at $x_{_0}$. By the definition of
$G$, it is easy to see that $\chi_{1\bar{1}}=\chi_{_1}$ is the
largest eigenvalue of $\{\chi_{i\bar{j}}\}$ at $x_{_0}$. Without
loss of generality, we can assume $t_{_0}>0$. Thus, locally, we
consider
 $H:=\log \chi_{1\bar{1}}-A(\varphi-\phi)$ instead, which also attains maximum at
$(x_{_0},t_{_0})$, with $H(x_{_0},t_{_0})=G(x_{_0},t_{_0})$. We
compute the evolution of $H$, namely the quantity $\frac{\p H}{\p
t}- F^{i\bar{j}}H_{i\bar{j}}$. Then at $(x_{_0},t_{_0})$, we have
\begin{equation}
\frac{\p H}{\p
t}=\frac{\chi_{1\bar{1},t}}{\chi_{1\bar{1}}}-A\frac{\p \varphi}{\p
t},\label{3.3}
\end{equation}
\begin{equation}
H_{i\bar{i}}=\frac{\chi_{1\bar{1},i\bar{i}}}{\chi_{1\bar{1}}}-\frac{|\chi_{1\bar{1},i}|^2}{\chi_{1\bar{1}}^2}-A(\varphi_{i\bar{i}}-\phi_{i\bar{i}}).\label{3.4}
\end{equation}
Take two derivatives along $\frac{\p}{\p z_1}$ direction to the
equation (\ref{flow}), one gets
\begin{equation}
\chi_{1\bar{1},t}=(\frac{\p \varphi}{\p t})_{1\bar{1}}=
\sum_{i=1}^{n}F^{i\bar{i}}\chi_{i\bar{i},1\bar{1}} + \sum_{1\leq
i,j,k,l\leq
n}F^{i\bar{j},\kl}\chi_{i\bar{j},1}\chi_{\kl,\bar{1}}.\label{3.7}
\end{equation}
Apply (\ref{flow}),(\ref{3.3}),(\ref{3.4}),(\ref{3.7}) we have, at
$(x_{_0},t_{_0})$
\begin{eqnarray*}
& &\frac{\p H}{\p t} - \sum_{i=1}^{n}F^{i\bar{i}}H_{i\bar{i}}\\
\notag &=&
\frac{1}{\chi_{1\bar{1}}}(\sum_{i=1}^{n}F^{i\bar{i}}\chi_{i\bar{i},1\bar{1}}
+ \sum_{1\leq i,j,k,l\leq
n}F^{i\bar{j},\kl}\chi_{i\bar{j},1}\chi_{\kl,\bar{1}})-A\frac{\p
\varphi}{\p
t}-\sum_{i=1}^{n}F^{i\bar{i}}H_{i\bar{i}}\\
\notag&=& \frac{1}{\chi_{1\bar{1}}}\sum_{i=1}^{n}
F^{i\bar{i}}(\chi_{i\bar{i},1\bar{1}}-\chi_{1\bar{1},i\bar{i}})-A\frac{\p
\varphi}{\p
t}+A\sum_{i=1}^{n}F^{i\bar{i}}(\varphi_{i\bar{i}}-\phi_{i\bar{i}})+B\\
\notag&=& \frac{1}{\chi_{1\bar{1}}}\sum_{i=1}^{n} F^{i\bar{i}}(\chi_{i\bar{i},1\bar{1}}-\chi_{1\bar{1},i\bar{i}})-A(c^{\frac{1}{k}}+F)+A\sum_{i=1}^{n}F^{i\bar{i}}(\chi'_{i\bar{i}}+\varphi_{i\bar{i}}-\phi_{i\bar{i}})-A\sum_{i=1}^{n}F^{i\bar{i}}\chi'_{i\bar{i}}+B\\
 \notag&=& \frac{1}{\chi_{1\bar{1}}}\sum_{i=1}^{n}
 F^{i\bar{i}}(\chi_{i\bar{i},1\bar{1}}-\chi_{1\bar{1},i\bar{i}})-Ac^{\frac{1}{k}}-2AF-A\sum_{i=1}^{n}F^{i\bar{i}}\chi'_{i\bar{i}}+B,
\end{eqnarray*}
where $$B=\frac{1}{\chi_{1\bar{1}}}\sum_{1\leq i,j,k,l\leq
n}F^{i\bar{j},\kl}\chi_{i\bar{j},1}\chi_{\kl,\bar{1}}+\sum_{i=1}^{n}F^{i\bar{i}}\frac{|\chi_{1\bar{1},i}|^2}{\chi_{1\bar{1}}^2}$$
includes all the third order derivatives terms of $\varphi$.

We claim that $B\leq0$, the proof of which we postpone to the end
of this section. By maximum principle, $\frac{\p H}{\p t} -
\sum_{i=1}^{n}F^{i\bar{i}}H_{i\bar{i}}\geq0$ at $(x_{_0},t_{_0})$,
thus
\begin{equation*}
\frac{1}{\chi_{1\bar{1}}}\sum_{i=1}^{n}F^{i\bar{i}}(\chi_{i\bar{i},1\bar{1}}-\chi_{1\bar{1},i\bar{i}})-Ac^{\frac{1}{k}}-2AF-A\sum_{i=1}^{n}F^{i\bar{i}}\chi'_{i\bar{i}}\geq0,
\end{equation*} i.e.
\begin{eqnarray}\label{3.5}
\frac{1}{\chi_{1\bar{1}}}\sum_{i=1}^{n}F^{i\bar{i}}(\chi_{i\bar{i},1\bar{1}}-\chi_{1\bar{1},i\bar{i}})&\geq&
A\sum_{i=1}^{n}F^{i\bar{i}}\chi'_{i\bar{i}}+Ac^{\frac{1}{k}}+2AF\\\notag
&\geq& A\sum_{i=1}^{n}F^{i\bar{i}}\chi'_{i\bar{i}}-A
c^{-\frac{1}{k}} F^2\\\notag
&=&A\sum_{i=1}^{n}F^{i\bar{i}}\chi'_{i\bar{i}}-Ac^{-\frac{1}{k}}\sigma_k^{\frac{2}{k}}(\chi^{-1}).
\end{eqnarray}
Notice that
\begin{equation*}\chi_{1\bar{1},i\bar{i}}=\chi_{i\bar{i},1\bar{1}}+\chi_{1\bar{1}}R_{1\bar{1},i\bar{i}}-\chi_{i\bar{i}}R_{i\bar{i},1\bar{1}},
\end{equation*}
so the left hand side of (\ref{3.5}) can be simplified as follows
\begin{eqnarray}\label{3.6}
\frac{1}{\chi_{1\bar{1}}}\sum_{i=1}^{n}F^{i\bar{i}}(\chi_{i\bar{i},1\bar{1}}-\chi_{1\bar{1},i\bar{i}})
& = &
\frac{1}{\chi_{1\bar{1}}}\sum_{i=1}^{n}F^{i\bar{i}}(\chi_{i\bar{i}}R_{i\bar{i}1\bar{1}}-\chi_{1\bar{1}}R_{1\bar{1}i\bar{i}})
\\\notag
&=&
\frac{1}{\chi_{1\bar{1}}}\sum_{i=1}^{n}F^{i\bar{i}}\chi_{i\bar{i}}R_{i\bar{i}1\bar{1}}-\frac{1}{\chi_{1\bar{1}}}\sum_{i=1}^{n}F^{i\bar{i}}\chi_{1\bar{1}}R_{1\bar{1}i\bar{i}}\\\notag
&\leq & \frac{-C_1F}{\chi_{1\bar{1}}} -\sum_{i=1}^{n} F^{i\bar{i}}
R_{1\bar{1}i\bar{i}}\\\notag &\leq &
\frac{C_0}{\chi_{1\bar{1}}}+C_2 \sum_{i=1}^{n} F^{i\bar{i}},
\end{eqnarray}
where $C_1=\max \{1,\sup_{i,j} \{R_{i\bar{i}j\bar{j}}\}\}$, $-C_2={ \min}\{-1,\inf_{i,j}
\{R_{i\bar{i}j\bar{j}}\} \}$ are upper and lower bound of
holomorphic bisectional curvature of $M$, and $C_0=C_1 \sup_{M}
[-F(\chi_{_0})]$. All constants here are positive.

Let $\chi_{1}\geq \cdots\geq\chi_{n}$ be the eigenvalues of $\chi$
with respect to $\omega$. Our goal is to get a uniform upper bound
for $\chi_{_1}=\chi_{1\bar 1}$.\\

If $k<n$, we have two cases: \\
Case 1. $\frac{\chi_{_1}}{\chi_{_n}}\leq N$. $N$ is the constant
in Theorem \ref{2.2}. From (\ref{3.2}), it follows that there exists a constant $C_3$
such that
$$ C_3\leq \sigma_k(\chi^{-1})\leq\frac{{n \choose
k}}{\chi_{_n}^k},$$
from which we get an upper bound
$$\chi_{_n}\leq({{n\choose k}\over C_{3}})^{1\over k}.$$ Hence $$\chi_{_1}\leq N\chi_{n}\leq C,$$ for some uniform constant $C$.\\
Case 2. $\frac{\chi_{_1}}{\chi_{_n}}\geq N$. Then by Theorem
\ref{2.2}, there exists $\epsilon>0$ such that
\be\label{3.9}\sum_{i=1}^{n}F^{i\bar{i}}\chi'_{i\bar{i}}-c^{-\frac{1}{k}}\sigma_k^{\frac{2}{k}}(\chi^{-1})\geq\epsilon\sum_{i=1}^{n}F^{i\bar{i}}\chi'_{i\bar{i}}.\ee
Since $\chi'$ is fixed and $M$ is compact, there exists $\gamma>0$,
such that \be\label{3.10}\epsilon
\sum_{i=1}^{n}F^{i\bar{i}}\chi'_{i\bar{i}}\geq
\gamma\sum_{i=1}^{n}F^{i\bar{i}}.  \ee Combine
(\ref{3.5}),(\ref{3.6}),(\ref{3.9}) and (\ref{3.10}), we get
\begin{equation}
\frac{C_0}{\chi_{1}}+C_2\sum_{i=1}^{n}F^{i\bar{i}}\geq
A\gamma\sum_{i=1}^{n}F^{i\bar{i}}.
\end{equation}
Since $\gamma >0$, we can choose $A$ so that $A\gamma- C_2=1$.
Hence,
\begin{equation} \label{3.11}\frac{C_0}{\chi_{1}}\geq \sum_{i=1}^{n}
F^{i\bar{i}}.
\end{equation} Apply G\r{a}rding's Inequality, Cauchy inequality and  (\ref{3.2}), we have
\begin{eqnarray}\label{p1}
\sum_{j=1}^{n}F^{j\bar{j}}&=&\sum_{j=1}^{n}\frac{1}{k}\sigma_k^{\frac{1}{k}-1}(\chi^{-1})\sigma_{k-1}(\chi^{-1}|j)\frac{1}{\chi_j^2}\\\notag
&\geq&\sigma_k^{\frac{1}{k}-1}(\chi^{-1})\sigma_k^{1-\frac{1}{k}}(\chi^{-1})\sigma_k^{\frac{1}{k}}(\chi^{-2})\\\notag
&\geq& \frac{\sigma_k^{\frac{2}{k}}(\chi^{-1})}{{n \choose
k}}\geq{C_{3}^{2\over k}\over{n \choose k}}.\\\notag
\end{eqnarray}
Combine (\ref{3.11}) and (\ref{p1}), we have
$$\chi_{1}\leq C,$$ for some constant $C$ depending only on
the initial data.\\

For $k=n$, notice in this case $c=1$. From Proposition~\ref{p21},
\begin{equation}
\sum_{i=1}^n F^{i\bar{i}}=
\frac{1}{n}\sigma_n^{-\frac{1}{n}}(\chi)\sum_{i=1}^n
\frac{1}{\chi_i}.
\end{equation}
By  (\ref{3.2}), there exists two positive
constants $C_4$ and $C_5$, such that
\begin{eqnarray}\label{3.14a}
0<C_4 \le \sigma_n^{-\frac{1}{n}}(\chi) \le C_5 <+\infty.
\end{eqnarray}
Now we can proceed directly from (\ref{3.5}) and (\ref{3.6}),
namely:
\begin{eqnarray}\label{3.14b}
A+ 2AF +A \sum_{i=1}^n F^{i\bar{i}}\chi'_{i\bar{i}}&\le&
\frac{C_0}{\chi_{1}}+C_1\sum_{i=1}^n F^{i\bar{i}}.
\end{eqnarray}
Assume $\chi'_{i\bar{i}} \ge \epsilon_o>0$.  Using (\ref{3.14a})
it follows that
\begin{eqnarray}\label{3.14c}
A- 2A\sigma_n^{-\frac{1}{n}}(\chi)
+\frac{A\epsilon_o}{n}\sigma_n^{-\frac{1}{n}}(\chi)
\sum_{i=1}^n\frac{1}{\chi_i} &\le&
\frac{C_0}{\chi_{1}}+C_6\sum_{i=1}^n\frac{1}{\chi_i}\le
C_7\sum_{i=1}^n\frac{1}{\chi_i}.
\end{eqnarray}
Apply  (\ref{3.14a}) again, we get
\begin{eqnarray}\label{3.14d}
(\frac{A\epsilon_o}{n}C_4-C_7) \sum_{i=1}^n\frac{1}{\chi_i}
&\le& 2AC_5.
\end{eqnarray}
Now we take $A$ such that  $\frac{A\epsilon_o}{n}C_4-C_7=1,\quad
i.e., \quad A=\frac{n(1+C_7)}{\epsilon_oC_4}$. From
(\ref{3.14d}), we have
$$\sum_{i=1}^n\frac{1}{\chi_i} \le C_8.$$ Since $\chi_{i}>0$
\begin{eqnarray}\label{3.14e}
\chi_{_i}\ge {C_8}^{-1},
\end{eqnarray}
Combining (\ref{3.14a}) and (\ref{3.14e}), it follows that there
exists a uniform  constant $C_9$
\be\label{3.14f}
\chi_{1}= {{(\Pi_{i=2}^{n}\chi^{-1}_{i})} \over {\sigma_{n}(\chi^{-1})}}\leq C_8^{n-1}/C_{4}^{n}=C,
\ee
for a uniform constant $C$.

In summary,  for all $1\leq k\leq n$, there exists a uniform constant $C$,
such that $\chi_1\leq C$. Back in the definition of G, we have
\begin{equation}
\log(\chi_{i\bar{j}})-A(\varphi-\phi)\leq\log(\chi_1(x_{_0}))-A(\varphi(x_{_0})-\phi(x_{_0})),
\end{equation}
so $$ \log (\chi_{i\bar{j}})\leq \log
C-A\varphi(x_0)+A\varphi+C'.$$ Exponentiating both sides, we get
the desired estimate.
\end{proof}

Now we prove the claim:
$B=\frac{1}{\chi_{1\bar{1}}}\sum_{i,j,k,l}F^{i\bar{j},\kl}\chi_{i\bar{j},1}\chi_{\kl,\bar{1}}+\sum_{i}F^{i\bar{i}}\frac{|\chi_{1\bar{1},i}|^2}{\chi_{1\bar{1}}^2}\leq
0$.
\begin{proof}
Case 1. $k<n$.\\

Recall from Proposition \ref{p21}, $F^{i\bar{j},\kl}$ is not zero
iff $i=j,k=l$ or $i=l,k=j$. According to the computation there, we
have for $i \neq j$
\begin{eqnarray}\notag
F^{i\bar{j},j\bar{i}}&=&
\frac{1}{k}(\frac{\sigma_{n-k}(\chi)}{\sigma_{n}(\chi)})^{\frac{1}{k}-1}(\frac{\sigma_n\sigma_{n-k-2}(\chi|i,j)-\sigma_{n-k}\sigma_{n-2}(\chi|i,j)}{\sigma_{n}^2})
\\\notag\label{3.12}
&=&-\frac{1}{k}(\frac{\sigma_{n-k}(\chi)}{\sigma_{n}(\chi)})^{\frac{1}{k}-1}(\frac{\chi_i\sigma_{n-k-1}(\chi|i,j)+\chi_j\sigma_{n-k-1}(\chi_j|i,j)+\chi_i\chi_j\sigma_{n-k-2}(\chi|i,j)}{\sigma_{n}^2})\\&<&0.
\end{eqnarray}
So we group terms as follows:\\

The first group:$$X=\frac{1}{\chi_{1\bar{1}}}( \sum_{1\leq i,j\leq
n}F^{i\bar{i},j\bar{j}}\chi_{i\bar{i},1}\chi_{j\bar{j},\bar{1}})+F^{1\bar{1}}\frac{|\chi_{1\bar{1},1}|^2}{\chi^2_{1\bar{1}}}\leq
0.$$

Let
$$f(\chi)=-(\frac{\sigma_{n-k}}{\sigma_n})^{\frac{1}{k}}(\chi).$$
It is sufficient to prove the following point-wise matrix
inequality:
\begin{eqnarray}
(f_{\chi_{i}\chi_{j}}+ \frac{f_{\chi_{i}}}{\chi_{j}}{\delta}_{ij}
)\le 0.\label{3.31} \end{eqnarray}

If we let $\lambda_i= \frac{1}{\chi_{i}}$, and $g(\lambda)=
{\sigma_{k}}^{\frac{1}{k}}(\lambda)$, then (\ref{3.31}) is
equivalent to the following
\begin{eqnarray}
(g_{\lambda_{i}\lambda_{j}}+
\frac{g_{\lambda_{i}}}{\lambda_{i}}{\delta}_{ij} )\ge
0,\label{3.32}
\end{eqnarray}
which is true by Proposition ~\ref{Prop2.a} and Remark
~\ref{remk2.0}.  See also Appendix for an alternative proof.

Second group: $$
Y=\frac{1}{\chi_{1\bar{1}}}\sum_{i=2}^{n}F^{i\bar{1},1\bar{i}}\chi_{i\bar{1},1}\chi_{1\bar{i},\bar{1}}+\sum_{i=2}^nF^{i\bar{i}}\frac{|\chi_{1\bar{1},i}|^2}{\chi_{1\bar{1}}^2}\leq
0.$$ The idea is to use $F^{i\bar{1},1\bar{i}}$ to control
$F^{i\bar{i}}$, take $i=2$ for example. By the K\"{a}hler property
of $\chi$, we have: $$
\chi_{i\bar{j},k}=\chi_{k\bar{j},i},
\chi_{i\bar{j},\bar{k}}=\chi_{i\bar{k},\bar{j}}.$$ It suffices
to show
$$\chi_{1\bar{1}}F^{j\bar{1},1\bar{j}}+F^{j\bar{j}}\leq 0, j\neq1.$$
After taking out the common factor
$\frac{1}{k\sigma_n^2(\chi)}(\frac{\sigma_{n-k}(\chi)}{\sigma_{n}(\chi)})^{\frac{1}{k}}$,
we are left to show
$$\chi_1[\sigma_{n}(\chi)\sigma_{n-k-2}(\chi|1,j)-\sigma_{n-k}(\chi)\sigma_{n-2}(\chi|1,j)]+\sigma_{n-k}(\chi)\sigma_{n-1}(\chi|j)-\sigma_{n-k-1}(\chi|j)\sigma_n(\chi)\leq 0.$$
Here we simply write $\chi_1$ for $\chi_{1\bar{1}}$. Use the
identity
$\sigma_k(\chi)=\sigma_k(\chi|1)+\chi_1\sigma_{k-1}(\chi|1)$, we
have
\begin{eqnarray}\label{y}\notag
&
&\chi_1[\sigma_{n}(\chi)\sigma_{n-k-2}(\chi|1,j)-\sigma_{n-k}(\chi)\sigma_{n-2}(\chi|1,j)]+\sigma_{n-k}(\chi)\sigma_{n-1}(\chi|j)-\sigma_{n-k-1}(\chi|j)\sigma_n(\chi)\\\notag
&=&\sigma_n(\chi)[\chi_1\sigma_{n-k-1}(\chi|j)-\sigma_{n-k-1}(\chi|j)]-\sigma_{n-k}(\chi)[\chi_1\sigma_{n-2}(\chi|1,j)-\sigma_{n-1}(\chi|j)]\\\notag
&=&-\sigma_n(\chi)\sigma_{n-k-1}(\chi|1,2)\leq 0.
\end{eqnarray}

The third group have all the remaining terms: $$
Z=\frac{1}{\chi_{1\bar{1}}}\sum_{1\leq i\leq n, 2\leq j\leq n,
i\neq
j}F^{i\bar{j},j\bar{i}}\chi_{i\bar{j},1}\chi_{j\bar{i},\bar{1}}\leq
0.$$

By (\ref{3.12}), each term in $Z$ is negative.

To sum up, we have $$B=X+Y+Z\leq0.$$

Case 2. $k=n$.\\

If we use the convention $\sigma_{-1}(\chi)=0$, the computation
above is  valid and can be simplied.
\end{proof}

\section{Convergence of the flow}
\setcounter{equation}{0} \setcounter{theorem}{0} In this section,
we study the properties of the functionals $
\tilde{\mathcal{F}}_{k,n}$ raised in the introduction, from which
we  prove the uniqueness of the solution of (\ref{euler}) and
$C^0$ estimate for the oscillation of $\varphi_t$. After getting
$C^0$ estimate of oscillation of $\varphi_t$, all the arguments in
\cite{W2} can be applied verbatim.

For any $\phi\in\mathcal{P}_{\chi_{_0}}$, let \be \delta
\mathcal{F}_k(\phi)=\int_M \delta \phi \chi_{\phi}^{k}\wedge
\omega^{n-k}, \label{4.1}\ee be the infinitesimal variation of the
functional $\mathcal{F}_k$. Then one has explicit formula for
$\mathcal{F}_k$:
$$\mathcal{F}_k(\phi)=\int_0^1\int_M \dot{\phi_t}
\chi_{\phi_t}^{k}\wedge\omega^{n-k} dt,$$ where $\phi_t$ is an
arbitrary path in $\mathcal{P}_{\chi_{_0}}$ connecting $0$ and
$\phi$, and $\dot{\phi_t}$ denotes time derivative. Then let: \be
\tilde{\mathcal{F}}_{k,n}(\phi)=\mathcal{F}_k(\phi)-c_{n-k}\mathcal{F}_n(\phi).\label{4.2}
\ee By the variational characterization of (\ref{4.1}), one has
\be \delta \tilde{\mathcal{F}}_{n-k,n}(\phi)=\int_M \delta \phi
(\chi_{\phi}^{n-k}\wedge \omega^k-c_{k}\chi{\phi^n}).\ee So
 the Euler-Lagrange equation of $\tilde{\mathcal{F}}_{n-k,n}$ is
 \be
 \chi_{\phi}^{n-k}\wedge \omega^k-c_{k}\chi_{\phi}^n=0,\label{4.4'}\ee which is exactly
(\ref{euler}). Regarding the second derivative of
$\tilde{\mathcal{F}}_{k,n}$, one chooses a path $\phi_t$ and use
(\ref{4.1}), (\ref{4.2}) to get: \begin{eqnarray}\notag \frac{d^2
\tilde{\mathcal{F}}_{n-k,n}(\phi_t)}{d t^2}&=& \int_M
\ddot{\phi_t}(\chi_{\phi}^{n-k}\wedge
\omega^k-c_{k}\chi_{\phi}^n)+
\int_M\dot{\phi_t}\partial\bar{\partial}\dot{\phi_t}((n-k)\chi_{\phi}^{n-k-1}\wedge
\omega^k-c_{k}n\chi_{\phi}^{n-1})\\
 &=& \int_M \ddot{\phi_t}(\chi_{\phi}^{n-k}\wedge
\omega^k-c_{k}\chi_{\phi}^n)+\int_M\partial \dot{\phi_t}\wedge
\bar{\partial}\dot{\phi_t}(c_{k}n\chi_{\phi}^{n-1}-(n-k)\chi_{\phi}^{n-k-1}\wedge
\omega^k) \label{4.5}
\end{eqnarray}

We observe the following
\begin{theorem} There is only
one critical point at the level of K\"{a}hler metric if such
critical point exists. \label{p4.1}
\end{theorem}
\begin{proof}
Suppose we have two critical points $\phi_{_0}$ and $\phi_{_1}$.
Consider the affine path $\phi_{t}=(1-t)\phi_{0}+t\phi_{1}$,
$t\in[0,1]$. $\phi_{_0}$ and $\phi_{_1}$ being critical points are
equivalent, in local coordinates, to following inequalities
$$\sigma_k(\chi_{\phi_{_0}}^{-1})=\sigma_k(\chi_{\phi_{_1}}^{-1})=c_k'.$$

Recall that in Section 2, we have proved
$-\sigma_k(\chi^{-1})=F$ is concave, which is equivalent to the convexity of $\sigma_k(\chi^{-1})$. Thus
$$ \sigma_k(\chi_{\phi_t}^{-1})\leq (1-t)c_k'+tc_{k}'=c_{k}', \ \ \ t\in[0,1].$$ Since
$\chi_{\phi_t}^{-1}$ is positive definite, we have
$\sigma_k(\chi_{\phi_t}^{-1}|i)<c_k'$. By
Proposition~\ref{new2.2}, it follows
$$c_{k}n\chi_{\phi_t}^{n-1}-(n-k)\chi_{\phi_t}^{n-k-1}>0$$ as a
$(n-1,n-1)$ form. Therefore by  (\ref{4.5}) and the facts that
$\dot{\phi_{_t}}=\phi_{_1}-\phi_{_0}$, $\ddot{\phi_{_t}}=0$, we conclude that
$\tilde{\mathcal{F}}_{n-k,n}(\phi_t)$ is a convex function:
$[0,1]\rightarrow \mathbb{R}$, with critical points at $t=0,1$.
This implies that $\tilde{\mathcal{F}}_{n-k,n}(\phi_t)$ is a
constant. Furthermore, the indentity
$$\frac{d^2
\tilde{\mathcal{F}}_{n-k,n}(\phi_t)}{d t^2}=0, $$  implies
$\dot{\phi_t}=\phi_{_1}-\phi_{_0}=C$ for some constant $C$, hence
$\chi_{\phi_{_0}}=\chi_{\phi_{_1}}$.

\end{proof}
Next, we establish some propositions regarding monotonicity of the
functionals which will lead to the $C^0$ estimates.

\begin{proposition} \label{p4.3}
The functional $\tilde{\mathcal{F}}_{n-k,n}$ is decreasing along
the flow (\ref{flow}).
\end{proposition}
\begin{proof}
We write (\ref{flow}) as $$
\dot{\varphi_t}=(c_k')^{\frac{1}{k}}+F,
$$ where
$F=-(\frac{\sigma_{n-k}(\varphi_t)}{\sigma_n(\varphi_t)})^{\frac{1}{k}}$.
\begin{eqnarray}\notag
\frac{d}{d t}\tilde{\mathcal{F}}_{n-k,n}(\varphi_t)&=& \int_M
\dot{\varphi_t} (\chi_{\varphi_t}^{n-k}\wedge
\omega^k-c_{k}\chi_{\varphi}^n)
\\&=&
\notag \frac{1}{{n \choose
k}}\int_M((c'_{k})^{1/k}+F)(F^k-c_k')\chi_{\varphi_t}^n\leq 0
\end{eqnarray}

The integrand is of the form $(a^{1/k}-b^{1/k})(b-a)$ which is
clearly non-positive.
\end{proof}

\begin{corollary}
Assume the convergence of the flow, i.e., the existence of the
solution of (\ref{euler}), then the global minimum of
$\tilde{\mathcal{F}}_{n-k,n}$ is realized by the critical metric.
\end{corollary}
\begin{proof}
It follows directly from Proposition \ref{p4.1} and Proposition
\ref{p4.3}.
\end{proof}

Towards $C^0$ estimate, we need another monotonicity: \bp
\label{p4.2}Let $\mathcal{F}_{n-k}$ defined as above, $\varphi_t$
the solution of flow (\ref{flow}), then
$$\frac{d \mathcal{F}_{n-k}(\varphi_t)}{d t}\leq 0,$$ i.e.\ $\mathcal{F}_{n-k}(\varphi_t)$ decreases along the flow.  In particular, $\mathcal{F}_{n-k}(\varphi_{t})\leq 0$ for all $t>0$.\ep
\proof   First we make following observation:
\begin{eqnarray}\label{4.3} \int_M \sigma_{n-k} dv &=& \int_M
(\frac{\sigma_{n-k}}{(\sigma_n)^{\frac{1}{k+1}}})(\sigma_n)^{\frac{1}{k+1}}
dv\\\notag &\leq & [\int_M
(\frac{\sigma_{n-k}}{(\sigma_n)^{\frac{1}{k+1}}})^{\frac{1+k}{k}}
dv]^{\frac{k}{k+1}}(\int_M \sigma_n dv)^{\frac{1}{1+k}}\\\notag
&=&(\int_M
\frac{(\sigma_{n-k})^{\frac{1+k}{k}}}{(\sigma_n)^{\frac{1}{k}}}
dv) ^{\frac{k}{k+1}} (\int_M \sigma_n dv)^{\frac{1}{1+k}}.
\end{eqnarray}
Recall $dv =\frac{\omega^n}{n!}$, so $\sigma_{n-k} dv = \frac{{n
\choose k}}{n!}\chi^{n-k}\wedge\omega^k$. So (\ref{4.3}) gives:
\begin{equation}\label{4.4}
\int_M (\frac{\sigma_{n-k}}{\sigma_n})^{\frac{1}{k}}
\chi^{n-k}\wedge\omega^k \geq c_k'^{\frac{1}{k}}
\int_M\chi^{n-k}\wedge\omega^k.
\end{equation}

Now we compute $\frac{d}{d t}\mathcal{F}_{n-k}(\varphi_t)$ by
choosing the path given by the flow then
\begin{eqnarray}\notag
\frac{d}{d t}\mathcal{F}_{n-k}(\varphi_t) & = &
\int_M\dot{\varphi_t}\chi_{\varphi_t}^{n-k}\wedge\omega^k \\
\notag&=& \int_M
[c_k'^{1/k}+F]\chi_{\varphi_t}^{n-k}\wedge\omega^k \\
\notag&=& \int_M
c_k'^{1/k}\chi_{\varphi_t}^{n-k}\wedge\omega^k-\int_M
(\frac{\sigma_{n-k}}{\sigma_n})^{\frac{1}{k}}\chi_{\varphi_t}^{n-k}\wedge\omega^k\leq
0.\\\notag
\end{eqnarray}\qed

From Proposition ~\ref{p4.2}, we know
$\mathcal{F}_{n-k}(\varphi_t)\leq0$. But the definition of
$\mathcal{F}_{n-k}$ is independent of the choice of the path, we
can choose the path $\gamma(s)=s\varphi_t$ to compute
$\mathcal{F}_{n-k}(\varphi_t)$ as well.
\begin{eqnarray*}
\mathcal{F}_{n-k}(\varphi_t)&=& \int_0^1\int_M \varphi_t
\chi_{s\varphi_t}^{n-k} \wedge\omega^k d s
\\\notag &=& \int_0^1\int_M \varphi_t
(s\chi_{\varphi_t}+(1-t)\chi_{_0})^{n-k}\wedge \omega^k d
s\\\notag &=& \sum_{l=0}^{n-k}\int_0^1{{n-k} \choose l} s^l
(1-s)^{n-k-l} d s \int_M \varphi_t \chi_{\varphi_t}^l\wedge
\chi_{_0}^{n-k-l} \wedge \omega^k\leq0.
\end{eqnarray*}

So at time $t$, we may write in short
$\mathcal{F}_{n-k}(\varphi_t)=\int_M \varphi_t d\mu_t$. Now we are
in the position to prove following:

\bt Suppose that
$\chi'^{n-1}-\frac{n-k}{n}\omega^k\wedge\chi'^{n-k-1}>0$. Let
$\varphi_t$ be a solution of (\ref{flow}) on $[0,\infty)$. Then
there exists a constant $\tilde{C}$, depending only on initial
data such that$$
\|\sup\varphi_t-\inf\varphi_t\|_{C^0}\leq\tilde{C}.$$\label{t4.1}
\et
\begin{proof}
It suffices to show a uniform lower bound of
$\inf\tilde{\varphi_t}$, where
$\tilde{\varphi_t}=\varphi_t-\sup_M\varphi_t$. Following
\cite{W2}, we prove by contradiction. If such a lower bound does
not exist, then we can choose a sequence of times $t_i\rightarrow
\infty$ such that
\begin{itemize}
    \item  $\inf_M \tilde{\varphi_{t_i}}
    =\inf_{t\in[0,t_i]}\inf_M\tilde{\varphi_t}$
    \item  $\inf_M\tilde{\varphi_{t_i}}\rightarrow -\infty$
\end{itemize}
Set $B=A/(1-\delta)$ where A is the constant in Theorem
\ref{t3.1}, and let $\delta$ be a small positive constant to be
determined later. Let $u=e^{-B\tilde{\varphi_{t_i}}}$. We apply
Lemma 3.3, Lemma 3.4 of \cite{W2}, there is a constant $c'$
independent of u,such that $$ \|u\|_{C^0}\leq C'\|u\|_{\delta}.$$
Since $u=e^{-B\tilde{\varphi_{t_i}}}$ and $\tilde{\varphi_{t_i}}$
satisfies $\sup_M \tilde{\varphi_{t_i}}=0$ and $$
\chi_{_0\kl}+(\tilde{\varphi_{t_i}})_{\kl}=\chi_{\kl}>0,$$ we can
apply Proposition 2.1 of \cite{T} to get a bound on
$\|u\|_{\delta}$ for $\delta$ small enough. This gives the uniform
$C^0$ estimate of $\tilde{\varphi_t}$.
\end{proof}

So far we have got the uniform $C^0$ estimate for oscillation of
$\varphi_t$, in order to get convergence we have to normalize
$\varphi_t$, namely let
$$\hat{\varphi_t}=\varphi_t-\frac{\mathcal{F}_{n-k}(\varphi_t)}{\int_M d\mu_t}.$$

Then $\hat{\varphi_t}$ takes value zero somewhere, by Theorem
\ref{t4.1}, $\|\hat{\varphi_t}\|_{C^0}\leq \tilde{C}$. With this
choice of normalization, we see the partial $C^2$ estimate is
actually uniform. By Theorem \ref{t3.1} $$ \|\partial
\bar{\partial}\hat{\varphi_t}\|_{C^0}=\|\partial
\bar{\partial}\varphi_t\|_{C^0}\leq A
e^{c(\varphi_t-\inf_{M\times[0,t]}\varphi_t)}.$$ For the exponent,
we have
\begin{eqnarray}
\varphi_t-\inf_{M\times[0,t]}\varphi_t&=&\hat{\varphi_t}+\frac{\mathcal{F}_{n-k}(\varphi_t)}{\int_M
d\mu_t}-\inf_{M\times[0,t]}(\hat{\varphi_t}+\frac{\mathcal{F}_{n-k}(\varphi_t)}{\int_M
d\mu_t})\\\notag
&\leq&\hat{\varphi_t}+\frac{\mathcal{F}_{n-k}(\varphi_t)}{\int_M
d\mu_t}-\inf_{M\times[0,t]}\hat{\varphi_t}-\inf_{M\times[0,t]}\frac{\mathcal{F}_{n-k}(\varphi_t)}{\int_M
d\mu_t}\\\notag
&=&\hat{\varphi_t}-\inf_{M\times[0,t]}\hat{\varphi_t}+\frac{\mathcal{F}_{n-k}(\varphi_t)}{\int_M
d\mu_t}-\inf_{M\times[0,t]}\frac{\mathcal{F}_{n-k}(\varphi_t)}{\int_M
d\mu_t}\\\notag
&=&\hat{\varphi_t}-\inf_{M\times[0,t]}\hat{\varphi_t}\leq
2\tilde{C}.
\end{eqnarray}
Last equality follows from Proposition \ref{p4.2} and the fact
$\int_M d \mu_t$ is independent of $t$. Hence, we have a uniform constant $C$ such that
$$||\partial\bar\partial \varphi_{t}||_{C^{0}}<C.$$

Since we get bound for complex hessian of $\varphi$, the underlying real  parabolic equation (\ref{flow}) has
uniform elliptic constants. By \cite{Wang1,Wang2}, one can deduce
$C^{2,\alpha}$ spatial and time estimate on $\varphi$. Then
classical Schauder theory can be applied to prove estimates all the way
to $C^{\infty}$. Consequently the flow exists on $[0,\infty)$. We
will provide more explanations of PDE aspect in Appendix B. \\

To show the convergence without passing to a subsequence, one can
follow the methods in \cite{C},\cite{W2}.

\section{Generalization and Applications}
\setcounter{equation}{0} \setcounter{theorem}{0} In this section,
we apply Theorem \ref{t3.1} to the product manifold $M\times C$,
where $C$ is an algebraic curve, to prove Theorem \ref{theorem2}.
\begin{proof}
First, let us recall the definition following constants: \be
c_k=c_{k,[\omega],[\chi]}=\frac{\int_M \chi_{_0}^k\wedge
\omega^{n-k}}{\int_M \chi_{_0}^n},\ee \be c_{k,\alpha}=c_k+\alpha
c_{k-1},\ \ \ \alpha\geq0,\label{5.1} \ee and cone condition
$\mathcal{C}_{k,\alpha}=\mathcal{C}_{k,\alpha}(\omega)$:
\begin{eqnarray} {\mathcal
C}_{k,\alpha}(\omega)&=&\{[\chi]\in\mathcal{H}^{+}, \ \exists
\chi'\in[\chi], \ {\rm such \ that}\\&&\notag
 c_{k,\alpha} n\chi'^{n-1}> (n-k)\chi'^{n-k-1}\wedge\omega^k+\alpha
 (n-k+1)\chi'^{n-k}\wedge\omega^{k-1}\}.\end{eqnarray}

Let $\omega_{_0}$, $\chi_{_0}\in [\chi]$ be two \kh \ forms on M,
$\omega_{_1}$ be a \kh \  form on $C$. Set $$
\tilde{\chi_{_0}}=\chi_{_0}+ a\omega_{_1},\ \ \ \text{and}\ \ \
\tilde{\omega_{_0}}=\omega_{_0}+\omega_{_1}.$$ Then on $M\times
C$, consider following flow in $\mathcal{P}_{\chi_{_0}}$,
\begin{equation}\label{flow2}
\frac{\p\varphi}{\p
t}=c^{\frac{1}{k}}-(\frac{\sigma_{n+1-k}(\tilde{\chi_{\varphi}})}{\sigma_{n+1}(\tilde{\chi_{\varphi}})})^{\frac{1}{k}},\
\ \ \ \  \varphi|_{t=0}=0,
\end{equation}
where
$\tilde{\chi_{\varphi}}=\tilde{\chi_{_0}}+\frac{\sqrt{-1}}{2}\p\bar{\p}\varphi$,
and $$c=\frac{\int_{M\times C}
\sigma_{n+1-k}(\tilde{\chi_{\varphi}})}{\int_{M\times C}
\sigma_{n+1}(\tilde{\chi_{\varphi}})}=\frac{a\int_M\sigma_{n-k}(\chi_{_0})+\int_M \sigma_{n-k+1}(\chi_{_0})}{a\int_M \sigma_n(\chi_{_0})}={n\choose k}c_k+\frac{1}{a}{n\choose k-1}c_{k-1}.$$\\

In local coordinates, one shall view the matrix
$((\tilde{\chi_{\varphi}})_{i\bar{j}})$ as $\begin{pmatrix}
  (\chi_{\varphi})_{i\bar{j}}& 0 \\
  0 & a \omega_1
\end{pmatrix}$.
In view of Theorem \ref{t3.1}, we want to bound the largest
eigenvalue of $\begin{pmatrix}
  (\chi_{\varphi})_{i\bar{j}}& 0 \\
  0 & a \omega_1
\end{pmatrix}$. Without loss of generality, we can assume the
corresponding direction is $\frac{\p}{\p z_1}\in T^{(1,0)}M$.
Otherwise the estimate follows trivially, since $\omega_1$ is
fixed under the flow. Compare the proof of Theorem \ref{t3.1}, we
impose condition:
\begin{equation}
\sigma_k(\tilde{\chi_{_0}}^{-1}|i)<c, \ \ \ \ \forall\
i=1,2,\cdots, n.\label{5.5}
\end{equation}
which translates to a condition on $M$ as:
\begin{equation}
\frac{1}{a}\sigma_{k-1}(\chi_{_0}^{-1}|i)+\sigma_k(\chi_{_0}^{-1}|i)<c,
 \ \ \ \ \forall\
i=1,2,\cdots, n.
\end{equation}
Then the whole argument applies. Moreover, $C^0$ estimate can be
applied directly. Therefore we get a stationary metric $\chi$ on
$M$ solving:
\begin{equation}
ac\chi^n=a{n\choose k}\chi^{n-k}\wedge \omega^k+{n\choose
k-1}\chi^{n-k+1}\wedge\omega^{k-1}.\label{5.7}
\end{equation}
After setting $\alpha=\frac{{n \choose k-1}}{a {n \choose k}}$,
one can readily check that $[\chi]\in \mathcal{C}_{k,\alpha}$
imply (\ref{5.5}), and (\ref{5.7}) becomes \be
c_{k,\alpha}\chi^n=\chi^{n-k}\wedge\omega^k+ \alpha
\chi^{n-k+1}\wedge \omega^{k-1}.\notag \ee
\end{proof}

Based on the known result, we can refine Chen's problem into the following:

\begin{conjecture} \label{conj}
For fixed $q$, $0\leq q \leq n$, and for any given
$\alpha=(\alpha_{0},\cdots,\alpha_{p})\in
{\mathbb{R}}^{p+1}$,$p\leq n-q$ $\alpha_{i}> 0$, $0\leq i\leq p$,
define
\begin{eqnarray*} c_{\alpha}&=&c_{k,\alpha,[\omega],[\chi]}=\sum_{i=0}^{p} c_{i+q}\alpha_{i} ,\\
\mathcal {\tilde F}_{\alpha,n}(\chi_{_0},\chi)&=& \sum_{i=0}^{p} \alpha_{i} \mathcal {\tilde F}_{i+q,n}(\chi_{_0},\chi),\\
 {\mathcal C}_{\alpha}(\omega)&=&\{[\chi]\in\mathcal{H}^{+}, \ \exists \chi'\in[\chi], \ {\rm such \ that}\
 c_{\alpha} n\chi'^{n-1}> \sum_{i=0}^{p} \alpha_{i}(n-i-q)\chi'^{n-i-q-1}\wedge\omega^{i+q}\}.
\end{eqnarray*}Then
\be
c_{\alpha}\chi_{\varphi}^n=\sum_{i=0}^{p}\alpha_i\chi_{\varphi}^{i+q}\wedge
\omega^{n-i-q},\ee
 has a unique smooth solution if and only if $[\chi]\in {\mathcal C}_{\alpha}(\omega)$; in this case, $\mathcal {\tilde F}_{\alpha,n}(\chi_{_0},\chi)$ obtains minimal at the given solution.
\end{conjecture}

Use the same method we can verify Conjecture \ref{conj} under some
additional conditions on $\alpha_i$'s. We consider $M\times
C_1\times C_2 \cdots \times C_p$, where $C_i$ are all algebraic
curves. Set $\omega_i$ be \kh \ forms on $C_i$. For $a_i>0$ set
$$
\tilde{\chi_{_0}}=\chi_{_0}+\sum_{i=1}^p a_i\omega_i,\ \
\tilde{\omega}=\sum_{i=0}^n \omega_i.$$ Follow the method above
one can solve \be c \sigma_{n+p}(\tilde{\chi})=
\sigma_{n+p-k}(\tilde{\chi}), \text{on}\  {\tilde M}:=M\times C_1\times C_2
\cdots \times C_p,\label {5.9}\ee where $c$ is the constant
satisfying
$$c=\frac{\int_{\tilde M} \sigma_{n+p-k}(\tilde{\chi})}{\int_{\tilde{M}} \sigma_{n+p}(\tilde{\chi})}.$$
Similarly, one reduces (\ref{5.9}) to an equation on $M$.
According to the relationship of $k$, $n$, and $p$, there will be
four cases which we state as a theorem.

\bt Let $M$, $\omega$, and $[\chi]$ be as above. $\Gamma_p$ is the
positive cone in $\mathbb{R}^{p}$. Conjecture~\ref{conj} holds for
the following special equations:

\begin{enumerate}
    \item For $p\geq k$ and $n>k$,
    $$c \chi^n= \beta_0 \chi^n + \beta_{1}\chi^{n-1}\wedge\omega
    +\cdots + \beta_k \chi^{n-k}\wedge\omega^k, c=\sum_{i=0}^{k}\beta_i c_i,$$
    for which we require the existence of a $b=(b_1,b_2,\cdots,b_p)\in \Gamma_p $ such that $\beta_i=\sigma_{k-i}(b){ n\choose i}, i=0,1,\cdots
    k$;
    \item For $p<k<n$,
    $$c \chi^n= \beta_0 \chi^{n+p-k}\wedge \omega^{k-p} +
    \beta_{1}\chi^{n+p-k-1}\wedge\omega^{k-p+1}
    +\cdots + \beta_p \chi^{n-k}\wedge\omega^k,c=\sum_{i=0}^{p}\beta_i c_{k-p+i},$$
    for which we require the existence of  a $b=(b_1,b_2,\cdots,b_p)\in \Gamma_p$ such that $\beta_i=\sigma_{p-i}(b){ n\choose k-p+i}, i=0,1,\cdots p$;
    \item For $p\geq k\geq n$,
    $$c \chi^n=\beta_0\chi^n+ \beta_{1}\chi^{n-1}\wedge\omega
    +\cdots + \beta_n \omega^{n},c=\sum_{i=0}^{n}\beta_i c_{i},$$
    for which we require the existence of a $b=(b_1,b_2,\cdots,b_p)\in \Gamma_p $ such that $\beta_i=\sigma_{k-i}(b){n\choose i},
    i=0,1,\cdots,n$;
    \item For $k>p$ and $k\geq n$,
    $$c \chi^n=\beta_{0}\chi^{n+p-k}\wedge\omega^{k-p}+
    \beta_{1}\chi^{n+p-k-1}\wedge\omega^{k-p+1}
    +\cdots + \beta_{n+p-k} \omega^{n},c=\sum_{i=0}^{n+p-k}\beta_i c_{k-p+i},$$
    where we require there exist some $b=(b_1,b_2,\cdots,b_p)\in \Gamma_p $ such that $\beta_i=\sigma_{p-i}(b){n\choose k-p+i},
    i=0,1,\cdots,n+p-k$.
\end{enumerate}\label{t5.1}
\et
\begin{remark}
It is due to our specific method that $\beta_i$'s have certain combinatorial constraints. We expect to remove these technical constraints in future works.
\end{remark}

We finish the discussion with a geometric application.

Consider $[\chi]=[\omega]+\epsilon [a]$, where $[a]\in H^{1,1}(M)$
and $\epsilon\in \mathbb{R}$ . Since $\omega$ is in the cone
$\mathcal{C}_{k}$, and the cone is obvious open, then for
$|\epsilon|$ small, $[\chi]\in\mathcal{C}_{k}$ for any
$k\in\{1,\cdots, n\}$. Thus, by Theorem~\ref{theorem1}, we have
$\chi\in[\chi]$ such that

$${{\chi^{n-k}\wedge \omega^{k}}\over{\chi^{n}}}=c_{k}.$$
On the other hand, it is easy to check that, on the manifold $M$,
we have the following point-wise inequalities: \be
{{\chi^{n-1}\wedge \omega}\over{\chi^{n}}}\geq{{\chi^{n-2} \wedge
\omega^{2}}\over{\chi^{n-1}\wedge \omega}}\geq
\cdots\geq{{\chi^{n-k}\wedge \omega^{k}}\over{\chi^{n-k+1}\wedge
\omega^{k-1}}}, \ee where any equality holds iff $\chi=
\lambda\omega$ for some constant $\lambda$. Thus, \be
{{\chi^{n-1}\wedge\omega}\over{\chi^{n}}}\geq
[{{\chi^{n-1}\wedge\omega}\over{\chi^{n}}}\cdot{{\chi^{n-2}
\wedge\omega^{2}}\over{\chi^{n-1}\wedge\omega^{1}}}\cdot
\cdots\cdot{{\chi^{n-k}\wedge\omega^{k}}\over{\chi^{n-k+1}\wedge\omega^{k-1}}}]^{1\over{k}}
=( c_{k})^{1\over{k}}. \ee This leads to \be
{\int_{M}{\chi^{n-1}\wedge\omega}\over{\int_{M}\chi^{n}}}\geq (
c_{k})^{1\over{k}}=[{\int_{M}{\chi^{n-k}\wedge
\omega^k}\over{\int_{M}\chi^{n}}}]^{1\over{k}}.\label{new2} \ee
Notice that (\ref{new2}) is independent of the choice of
$\chi\in[\chi]$. Notice $[\chi]=[\omega]+\epsilon [a]$. Take
$k=2$, and expand both sides of (\ref{new2}) as a series of
$\epsilon$, then let $\epsilon \rightarrow 0$, we get the
following inequality:
\begin{equation}(\int_{M}\omega^{n-2}\wedge a^{2})(\int_{M}\omega^{n})\leq
\frac{n-1}{n(n-2)}(\int_{M}\omega^{n-1}\wedge a)^{2},
\end{equation} where the identity holds iff $[a]=\lambda'[\omega]$ for
some constant $\lambda'$. This is exactly the Riemann-Hodge
bi-linear relation for $(1,1)$-classes (see, e.g.,~\cite{GH}).

\bigskip

  \renewcommand{\theequation}{A.\arabic{equation}}
  \renewcommand{\thetheorem}{A.\arabic{theorem}}

  % redefine the command that creates the equation no.
  \setcounter{equation}{0}  % reset counter
  \setcounter{theorem}{0}
  \section*{APPENDIX A}  % use *-form to suppress numbering
  \bigskip

In this appendix, we first present another proof of Remark
~\ref{remk2.0}. For the convenience of readers, we restate it as
the following: \bp\label{a1} Let $g=\sigma_k^{\frac{1}{k}}(\chi)$,
and $\chi\in \Gamma_n$. Let $g_i:=\frac{\p g}{\p \chi_i},
g_{ij}:=\frac{\p^2 g}{\p \chi_i \p \chi_j}$. Then the matrix
$g_{ij}+\frac{g_i}{\chi_j}\delta_{ij}$ is nonnegative. \ep \proof

Step 1.\\
Consider $h:=\sigma_k(\chi^{\frac{1}{k}})$. Use the same notation
as above, we claim
$$h_{ij}+\frac{h_i}{\chi_j}\delta_{ij}\geq 0.$$ Direct
computation shows that:
\begin{eqnarray}\label{app1}
h_i&=&\frac{1}{k}\sigma_{k-1}(\chi^{\frac{1}{k}}|i)\chi_i^{\frac{1}{k}-1},\\\label{app2}
h_{ij}&=&\frac{1}{k^2}\sigma_{k-2}(\chi^{\frac{1}{k}}|i,j)\chi_i^{\frac{1}{k}-1}\chi_j^{\frac{1}{k}-1}+\frac{1}{k}(\frac{1}{k}-1)\sigma_{k-1}(\chi^{\frac{1}{k}}|i)\chi_j^{\frac{1}{k}-2}\delta_{ij}.
\end{eqnarray}
Introduce the following notation: for $I=(i_1,i_2,\cdots i_l)$ an
arbitrary index set of length $l$, let $\sigma_{k;I}:=\sum
_{|I|=k}\chi_{I}\sigma_{k-l}(\chi|I)$,where
$\chi_{I}=\chi_{_{i_1}}\chi_{_{i_2}}\cdots\chi_{_{i_l}}$.
Basically, it is the collection of terms in which indices $i\in I$
appear. In this notation, we can rewrite (\ref{app1}),
(\ref{app2}) as:
\begin{eqnarray}
h_i&=&\frac{\sigma_{k;i}}{k\chi_i},\\
h_{ij}&=&\frac{\sigma_{k;i,j}}{k^2\chi_i\chi_i}, \ \ \ \text{for} \
i\neq
j,\\
h_{ii}&=&\frac{1}{k}(\frac{1}{k}-1)\frac{\sigma_{k;i}}{\chi_i^2}.
\end{eqnarray}
So $h_{ij}+\frac{h_i}{\chi_j}\delta_{ij}$ equals:
\begin{equation}\begin{bmatrix}
  \frac{\sigma_{k;1}}{k^2\chi_1^2} & \frac{\sigma_{k;1,2}}{k^2\chi_1\chi_2}& \cdot &\cdot& \frac{\sigma_{k;1,n}}{k^2\chi_1\chi_n} \\
 \frac{\sigma_{k;1,2}}{k^2\chi_1\chi_2} & \frac{\sigma_{k;2}}{k^2\chi_2^2}&  &  &  \\
   \cdot&  & \cdot & &  \\
   \cdot&  &&  \cdot& \\
  \frac{\sigma_{k;1,n}}{k^2\chi_1\chi_n} & \cdot &  & & \frac{\sigma_{k;n}}{k^2\chi_n^2} .\\
\end{bmatrix}
\end{equation}
Then it is equivalent to show that $$A:=\begin{bmatrix}
 \sigma_{k;1} & \sigma_{k;1,2} &  \cdot&  & \sigma_{k;1,n} \\
 \sigma_{k;1,2}& \sigma_{k;2} &  &  &  \cdot\\
   \cdot &  & \cdot &  &  \\
  &  & &  \cdot & \\
 \sigma_{k;1,n}&  &  &  & \sigma_{k;n} \\
\end{bmatrix}$$ is nonnegative. For an index set $I$, Let $E_{I}$ be the matrix
having entry $1$ in $i$-th row and $j$-th column of an $n\times n$ matrix, where $i,j\in I$, and entry $0$
elsewhere. It is clear that $E_{I}$ is nonnegative. Moreover,
we have the following nice decomposition:
\begin{equation}
A=\sum_{|I|=k}\chi_{I}E_{I}\geq 0.
\end{equation}
 Thus
$$h_{ij}+\frac{h_i}{\chi_j}\delta_{ij}\geq 0.$$\smallskip
Step 2.\\
We claim
$$h_{ij}+\frac{h_i}{\chi_j}\delta_{ij}-\frac{h_ih_j}{h}\geq 0.$$
\smallskip
We use a nice trick due to Andrews~\cite{Andrews}.
Since $h$ is homogenous of degree $1$,  $h_i\chi_i=h$.
Differentiate both sides, one gets $h_{ij}\chi_i=0$. Consequently,
\begin{equation}
(h_{ij}+\frac{h_i}{\chi_j}\delta_{ij}-\frac{h_ih_j}{h})\chi_i\chi_j=0,
\end{equation}
i.e. $\chi$ is a null vector. In order to show
$h_{ij}+\frac{h_i}{\chi_j}\delta_{ij}-\frac{h_ih_j}{h}\geq 0$, one
then only need to look at a subspace transversal to the null
vector $\chi=(\chi_1,\cdots, \chi_n).$ Naturally, we choose the
subspace defined by $\{\xi|h_i\xi_i=0\}$. Then
$(h_{ij}+\frac{h_i}{\chi_j}\delta_{ij}-\frac{h_ih_j}{h})\xi_i\xi_j=(h_{ij}+\frac{h_i}{\chi_j}\delta_{ij})\xi_i\xi_j$,
which is nonnegative from step 1. \\
Step 3.\\
$g(\chi_1, \cdots \chi_n)=h^{\frac{1}{k}}(\chi_1^k, \cdots,
\chi_n^k)$, a simple computation shows that:
\begin{equation}\label{app3}
g_{ij}+\frac{g_i}{\chi_j}\delta_{ij}=kh^{\frac{1}{k}-1}(\lambda)\lambda_i^{1-\frac{1}{k}}\lambda_j^{1-\frac{1}{k}}[h_{ij}+\frac{h_i}{\lambda_j}\delta_{ij}-\frac{k-1}{k}\frac{h_i
h_i}{h}],
\end{equation}
where $\chi_i^k=\lambda_i$. Thus,
\begin{equation}
h_{ij}+\frac{h_i}{\lambda_j}\delta_{ij}-\frac{k-1}{k}\frac{h_i
h_i}{h}\geq h_{ij}+\frac{h_i}{\lambda_j}\delta_{ij}-\frac{h_i
h_i}{h}\geq 0 , \end{equation} the last inequality is due to step 2. The proof is thus
completed. \qed
\begin{remark}It is clear from the above proof that the conclusion of Proposition~\label{a1} holds for $g=\sigma_k^{\epsilon}(\chi)$, with $\epsilon>0$.
\end{remark}

\bigskip

  \renewcommand{\theequation}{B.\arabic{equation}}
  \renewcommand{\thetheorem}{B.\arabic{theorem}}

  % redefine the command that creates the equation no.
  \setcounter{equation}{0}  % reset counter
  \setcounter{theorem}{0}
  \section*{APPENDIX B}  % use *-form to suppress numbering
  \bigskip

In this appendix, we summarize the classical parabolic
Krylov-Evans theory that are applied in this paper. In particular, we deduce time
$C^{\frac{\alpha}{2}}$ estimates for $\partial \bar{\partial}
\varphi$ for (\ref{flow}). These estimates are local in nature.
This proof is essentially  due to Lihe Wang~\cite{Wang2}.

In the parabolic case, it is also convenient to
introduce the following regularity notation. We say $\varphi=\varphi(x,t) \in
C^{2, \alpha}$ in the parabolic sense if and only if $\varphi \in
C^{2,\alpha}$ in spatial variable $x\in \mathbb{R}^{n}$, and $\varphi\in C^{1,
\frac{\alpha}{2}}$ in time variable $t\in \mathbb{R}$ in the usual sense. Different
regularity is due to different scaling of spatial and time
variables. We will also write $C^{2+\alpha,
1+\frac{\alpha}{2}}_{x,t}$ to indicate regularity respectively. For a thorough exposition, we refer readers to \cite{Wang1} and \cite{Wang2}.\\

The fundamental tool to attack nonlinear parabolic equation is
following:
\begin{theorem}[Krylov-Safanov]
Let $\varphi$ be a solution of \begin{equation}
\varphi_t=a_{ij}(x,t)\varphi_{ij},
\end{equation} in $\mathbf{Q}_1$, and $a_{ij}$ is uniform elliptic, then $\varphi$
is in $C^{\alpha}_{loc}(\mathbf{Q}_1)$, i.e., $\varphi$ is
$C^{\alpha}$ in spatial and $\varphi$ is $C^{\frac{\alpha}{2}}$ in
time.
\end{theorem}

The parabolic equation we have is:
\begin{equation} \label{ape}
\frac{\partial \varphi}{\partial t}= c+ F(\partial\bar{\partial}
\varphi).
\end{equation}
By Theorem \ref{t3.1}, $F$ is a uniform elliptic, concave
operator. Taking derivative with respect to $t$ both sides of
(\ref{ape}), one has \be
\varphi_{tt}=F^{i\bar{j}}(\varphi_t)_{i\bar{j}}. \ee By
Theorem~B1, $\varphi_t$ is $C^{\alpha}_{x}$. Thus
(\ref{ape}) can be viewed as an elliptic
equation. Then by the elliptic Krylov-Evans theory, one has spatial $C^{\alpha}$
estimate on $D^{2}_{x}\varphi$. To have
$C^{\alpha}$ estimate for $D^{2}_{x}\varphi$, it is sufficient to show time
$C^{\frac{\alpha}{2}}$ estimate. Since the problem is local in
nature, we just need to prove time $C^{\frac{\alpha}{2}}$ estimate
at $(0,0)$.

Since $\varphi$ is $C^{2,\alpha}$ in spatial, there exist two quadratic polynomials $P_t(x)$ and
$P_0(x)$ such that
\begin{equation}\label{ap1}
|\varphi(x,t)-P_t|\leq C|x|^{2+\alpha},\ \  |x|\leq \sqrt t,
\end{equation}
\begin{equation}\label{ap2}
|\varphi(x,0)-P_0|\leq C|x|^{2+\alpha}, \ \ |x|\leq \sqrt t.
\end{equation} Also, since  $\varphi_t\in C^{\alpha}$
\begin{equation}\label{ap3}
|\varphi(x,t)-\varphi(x,0)-t \varphi_t(x,0)|\leq C
t^{1+\frac{\alpha}{2}}, \ \ |x|\leq \sqrt t.
\end{equation}
\begin{equation} \label{ap5}
|\varphi_t(x,0)-\varphi_t(0,0)|\leq C |x|^{\alpha}.
\end{equation}
By (\ref{ap1}),(\ref{ap2}) and (\ref{ap3}) together, we have
\begin{equation}\label{ap4}
|P_t(x)-P_0(x)-t\varphi_t(x,0)|\leq C t^{1+\frac{\alpha}{2}}.
\end{equation}

(\ref{ap4}) and (\ref{ap5}) imply that
\begin{equation}
|P_t(x)-P_0(x)|\leq C t^{1+\frac{\alpha}{2}},  \ \ |x|\leq \sqrt t.
\end{equation}
For a quadratic polynomial, one has
\begin{equation}
||D_{x}^2 P||_{B_{r}}\leq C \frac{||P||_{L^{\infty}(B_{r})}}{r^2}.
\end{equation}
Therefore, \begin{equation} ||D_{x}^2 P_t-D_{x}^2 P_0||_{B_{\sqrt t}}\leq C
\frac{ ||P_t-P_0||_{L^{\infty}(B_{\sqrt{t}})}}{t}\leq C t^{\frac{\alpha}{2}}.
\end{equation} which implies that
$$||D^{2}_{x}\varphi(0,t)-D^{2}_{x}\varphi(0,0)||\leq C t^{\frac{\alpha}{2}}.$$

\end{document}